\documentclass[12pt]{article}
\usepackage[cm]{fullpage}
\usepackage{amssymb,amsmath,amsthm,amscd,graphicx}
\usepackage{subcaption}
\usepackage[labelsep=quad,indention=20pt]{caption}

\theoremstyle{definition}
\newtheorem{example}{Example}[section]
\DeclareMathOperator{\Tr}{Tr}
\def\sp{\mathop{\rm sp}}
\def\dist{\mathop{\rm dist}}
\newtheorem{algorithm}{Algorithm}

\title{Dynamic isoperimetry and the \\ geometry of Lagrangian coherent structures}

\author{Gary Froyland \\ \\ School of Mathematics and Statistics \\ University of New South Wales \\ Sydney NSW 2052, Australia}
\begin{document}
\maketitle
\begin{abstract}
The study of transport and mixing processes in dynamical systems is particularly important for the analysis of mathematical models of physical systems.
We propose a novel, direct geometric method to identify subsets of phase space that remain strongly coherent over a finite time duration.
This new method is based on a dynamic extension of classical (static) isoperimetric problems;  the latter are concerned with identifying submanifolds with the smallest boundary size relative to their volume.

The present work introduces \emph{dynamic} isoperimetric problems;  the study of sets with small boundary size relative to volume \emph{as they are evolved by a general dynamical system}.
We formulate and prove dynamic versions of the fundamental (static) isoperimetric (in)equalities;  a dynamic Federer-Fleming theorem and a dynamic Cheeger inequality.
We introduce a new dynamic Laplace operator and describe a computational method to identify coherent sets based on eigenfunctions of the dynamic Laplacian.

Our results include formal mathematical statements concerning geometric properties of finite-time coherent sets, whose boundaries can be regarded as Lagrangian coherent structures.
The computational advantages of our new approach are a well-separated spectrum for the dynamic Laplacian, and flexibility in appropriate numerical approximation methods.
Finally, we demonstrate that the dynamic Laplace operator can be realised as a zero-diffusion limit of a newly advanced probabilistic transfer operator method  \cite{F13} for finding coherent sets, which is based on small diffusion.
Thus, the present approach sits naturally alongside the probabilistic approach \cite{F13}, and adds a formal geometric interpretation.
%
%
%

%
\end{abstract}

\def\D{\mathcal{D}_{\epsilon}}
\def\Dx{\mathcal{D}_{X,\epsilon}}
\def\Dy{\mathcal{D}_{Y'_\epsilon,\epsilon}}
\def\Dxd{\mathcal{D}_{\dot{X},\epsilon}}
\def\Dyd{\mathcal{D}_{\dot{Y}'_\epsilon,\epsilon}}
\def\P{\mathcal{P}}
\def\Pd{\dot{\mathcal{P}}}
\def\Pe{\mathcal{P}_\epsilon}
\def\Le{\mathcal{L}_\epsilon}
\def\Leb{{\rm Leb}}
\def\Led{\dot{\mathcal{L}}_\epsilon}
\def\ax{\alpha_{X,\epsilon}}
\def\ay{\alpha_{Y,\epsilon}}
\def\span{{\rm span}}
\def\Pr{{\rm Pr}}

\newtheorem{theorem}{Theorem}[section]
\newtheorem{lemma}[theorem]{Lemma}
\newtheorem{corollary}[theorem]{Corollary}
\newtheorem{sublemma}[theorem]{Sublemma}
\newtheorem{remark}[theorem]{Remark}
\newtheorem{definition}{Definition}

\section{Introduction}

The study of Lagrangian coherent structures in nonlinear dynamics is broadly concerned with the identification of spatial structures in phase space that behave in a relatively ``stable'' way under the dynamics by resisting high levels of distortion and/or diffusion.
In the case of purely advective dynamics governed by a nonlinear map or time-dependent ordinary differential equations, if the structure is a full-dimensional set\footnote{Frequently, \emph{coherent structures} are co-dimension 1 objects, while full-dimensional objects are called \emph{coherent sets}.}, this set resists filamentation under the nonlinear dynamics and the ratio of boundary size to the volume of the set remains relatively unchanged.
When the dynamics is a combination of advection and diffusion, for example, a time-dependent Fokker-Planck equation, a coherent set resists mixing with the surrounding phase space, again through the mechanism of retaining a relatively low boundary size.
There is a long history of development of related ideas spread across the dynamical systems, fluid dynamics, and geophysics literature.
We mention just two early related works: \cite{pierrehumbert_yang}, which contains several ideas concerning mixing and transport mitigation in fluids, and the book \cite{ottino}, which discusses purely advective (chaotic) mixing.
These ideas, and the theory and algorithms developed subsequently, have grown into their own field, and have been employed across a wide spectrum of physical, biological, environmental, and engineering applications.

The study of sets with minimal boundary is known in differential geometry as an isoperimetric problem.
The classic isoperimetric problem in $\mathbb{R}^2$ is to determine the set $S$ with least boundary length (perimeter), given a fixed area;  or equivalently to find a set $S$ of fixed (iso-) perimeter with greatest area.
The unique solution of this problem is a disk; all other sets $S$ satisfy the inequality $\mbox{length}(\partial S)/\sqrt{\mbox{area}(S)}>2\sqrt{\pi}$, an example of an \emph{isoperimetric inequality}.
The obvious generalisation of this problem in $\mathbb{R}^d$ is true: $d$-balls minimise surface area and
$\ell_{d-1}(\partial S)/\ell_d(S)^{1-1/d}>d{\omega_d}^{1/d}$, where ${\omega_d}$ is the volume of a unit ball in $\mathbb{R}^d$, and $\ell_d, \ell_{d-1}$ are $d$ and $d-1$-dimensional volume (see e.g.\ \cite{chavelisoperimetric}).
Because we have in mind applications to fluid flow, we focus on compact domains $M$ rather than  $\mathbb{R}^d$.
For compact, connected $M$, one has a hypersurface $\Gamma\subset M$ disconnect $M$ into two pieces $M_1,M_2$, just as the $d$-ball disconnects $\mathbb{R}^d$.
One tries to find a disconnecting hypersurface $\Gamma$ that minimises the ratio $$\mathbf{h}(\Gamma):=\ell_{d-1}(\Gamma)/\min\{\ell_d(M_1),\ell_d(M_2)\}.$$


One of our main contributions is to develop a theory of \emph{dynamic} isoperimetry, where one studies the \emph{evolution} of hypersurfaces $\Gamma\subset M$ that disconnect phase space $M$, under a nonlinear transformation $T:M\to T(M)$.
Both the manifold $M$ and the separating surface $\Gamma$ are subjected to general nonlinear dynamics, representing the action of some (possibly chaotic) flow over some finite-time duration.
The solution to this dynamic isoperimetric problem may have nothing to do with the solution to the static problem because even if $\Gamma$ has low co-dimension 1 volume, the size of $T(\Gamma)$ may be greater, and if $T$ is chaotic, $T^k(\Gamma)$ for a modest number of iterations $k$ may have significantly greater size (see Figure \ref{coherencefig}b and \ref{coherencefig}c).

\begin{figure}[hbt]
\centering
  \includegraphics[width=12cm]{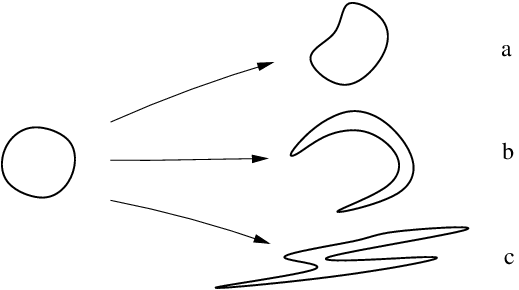}\\
  \caption{The two-dimensional set on the left with boundary $\Gamma$ has a low boundary size to volume ratio.  The three sets $T(\Gamma)$ on the right are three possible images of the shape on the left under three different nonlinear volume-preserving dynamical systems $T$ over a fixed finite time duration.  Under dynamics `a', the set on the left retains a low boundary size to volume ratio, but under dynamics `b' and `c', the boundary size is significantly increased.}\label{coherencefig}
\end{figure}

Thus, the dynamics plays a key role in the selection of a surface $\Gamma$ that remains small relative to the domain volume when evolved under a general volume-preserving nonlinear dynamical system.
Clearly such surfaces $\Gamma$ bound sets that are very natural candidates for {coherent sets}.
To take a real-world example, coherent sets such as oceanic eddies retain water mass
and remain coherent by their boundary remaining as short as possible over an extended period of time.
This reduces diffusion through the eddy boundary $T^k\Gamma$ at times $k=0,1,2,\ldots$ via small-scale diffusion processes (e.g. \cite{balasuriyajones99}).


In discrete time, under a single application of $T$, this motivates the minimisation of the quantity
\begin{equation}
\label{cheegereqn_intro}
\mathbf{h}^D(\Gamma):=\frac{\ell_{d-1}(\Gamma)+\ell_{d-1}(T(\Gamma))}{2\min\{\ell(M_1),\ell(M_2)\}},
\end{equation}
where $\Gamma$ varies over smooth hypersurfaces disconnecting $M$ into two connected pieces $M_1, M_2$.
In continuous time, we consider smooth flow maps $T^{(t)}:M\to T^{(t)}(M)$, $t\in[0,\tau]$ and the quantity
\begin{equation}
\label{hdynt_intro}
\mathbf{h}_{[0,\tau]}^{D}(\Gamma):=\frac{\int_0^\tau \ell_{d-1}(T^{(t)}\Gamma)\ dt}{\tau\min\{\ell(M_1),\ell(M_2)\}}.
\end{equation}
In addition to the geometric interpretation of Figure \ref{coherencefig}, the expression (\ref{hdynt_intro}) is also directly proportional to the mass lost through the boundary over the finite time interval $[0,\tau]$ via continually-present small-scale diffusion.
This latter interpretation motivates the additive combination of terms in (\ref{cheegereqn_intro}) and (\ref{hdynt_intro}) (as opposed to e.g.\ a multiplicative combination).

We focus on the difficult setting of general time-dependent dynamics.
In the autonomous dynamics setting, classical ``coherent'' (in fact, invariant) objects such as invariant tori or invariant cylinders are completely invariant under the dynamics and thus their boundaries remain fixed and unchanging.
Furthermore, trajectories that begin on the inside of these structures can never leave through their co-dimension 1 boundary.
Thus, these objects can be regarded as ``ideal'' coherent structures  and are relatively well-understood.
In the general time-dependent dynamics setting, the existence of such perfectly invariant structures is highly unlikely.
\begin{definition}
\label{def:ftcs}
For the finite-time dynamics considered we define a \emph{maximally coherent structure} on $M$ to be a minimizing $\Gamma$ for (\ref{cheegereqn_intro}) or (\ref{hdynt_intro}) (see also Section \ref{sect:multistep}) when the infimum is achieved;  otherwise we select a $\Gamma$ for which $\mathbf{h}^D(\Gamma)$ is arbitrarily close to the infimum.
The corresponding \emph{maximally coherent set} is defined to be the $M_k$, $k=1,2$ with minimal volume arising from the disconnection $\Gamma$.
\end{definition}
Existing approaches to identifying coherent structures fall broadly into two categories: probabilistic methods and geometric methods.
Probabilistic approaches to finding coherent structures are based around the {transfer operator} (or Perron-Frobenius operator) and can be applied to systems with a combination of advection and diffusion, or purely advective dynamics.
These methods look for \emph{finite-time coherent sets} \cite{FSM10,F13,FPG14}: sets that resist mixing with the rest of phase space and represent global transport barriers to complete mixing.
These constructions have found application in atmospheric dynamics to map and track the Antarctic polar vortex \cite{FSM10}, and in ocean dynamics to track an oceanic eddy in the Agulhas current \cite{FHRSS12,FHRvS15}, in both two and three dimensions.
In the purely advective setting, the constructions underpinning the transfer operator methods rely on small amounts of diffusion \cite{F13}.
This small diffusion makes complete phase space mixing possible and is required for other technical reasons;  these points are discussed in \cite{F13} and in the present work in Section \ref{sect:zerolimit}.
Further, the boundary sizes of finite-time coherent sets are implicitly measured in \cite{F13} because the localised diffusion can only eject mass near the boundaries of coherent sets;  thus the mixing experienced over a finite time is tied to the boundary sizes of the coherent sets.
We therefore expect our proposed dynamic isoperimetry methodology to be compatible with \cite{F13} and in fact, we show that the former arises as a zero-diffusion limit of the latter, which makes explicit the geometry contained in the probabilistic methods for small diffusion.

In recent years there have been several geometric methods proposed to characterise either trajectories or co-dimension 1 surfaces that represent coherent structures in purely advective dynamics \cite{mancho_M,mezicmeso,Haller_11,Haller_12,allshouse_thiffeault,ma_bollt_shape,romkedar_M}.
In two dimensions, \cite{Haller_11} defines a hyperbolic LCS as a material curve that has repelling dynamics normal to the curve in forward time and greater repulsion than nearby curves, while \cite{Haller_12} defines transport barriers as curves that are local minimisers of length functionals integrated over a finite time interval, with various hyperbolic, shear, and elliptic boundary conditions for the associated Euler-Langrange equations.
The approach \cite{ma_bollt_shape} suggests that curves formed from points that experience local rigid-body motion over a finite-time duration are associated with coherent dynamics.
Most approaches compute various scalar fields from Lagrangian trajectories and infer corresponding dynamic properties from the fields.


The main contributions of this paper include formulations of dynamic versions of classical objects in isoperimetric theory and formulations and proofs of dynamic versions of fundamental isoperimetric theorems.
We formulate a dynamic version of the \emph{Cheeger constant} $\mathbf{h}$ (the minimal ratio of the $d-1$-dimensional volume of a disconnecting hypersurface $\Gamma$ to the disconnected volumes of $M_1, M_2$) and the Sobolev constant (a functional representation of the Cheeger constant, where the separating hypersurface $\Gamma$ is the level set of a smooth function).
The celebrated \emph{Federer-Fleming theorem} equates these two constants, formally linking geometric and functional representations of these notions of isoperimetry.
We formulate and prove a dynamic version of the Federer-Fleming theorem, linking our new dynamic Cheeger and Sobolev constants (Section \ref{sect:ff}).

We further formulate and prove a dynamic version of the \emph{Cheeger inequality}, which relates the second largest eigenvalue of the Laplace operator $\triangle$ on $M$ to the Cheeger constant (Section \ref{sect:cheegerineq}).
This requires a replacement of the Laplace operator with a new operator that incorporates the general nonlinear dynamics.
In the discrete-time volume-preserving setting, the operator on $M$ corresponding to the expression (\ref{cheegereqn_intro}) is \begin{equation}
\label{dynlap_intro}
(\triangle+\mathcal{P}^*\triangle\mathcal{P})/2,
\end{equation}
where $\mathcal{P}f=f\circ T^{-1}$ is the transfer operator for $T$.
We develop a spectral theory for this new operator in Section \ref{sect:dynlapspec}, and propose an algorithm that uses eigenvectors of this operator to identify coherent sets in practice.
In Section \ref{sect:zerolimit} we demonstrate that one can recover the new dynamic Laplace operator from the probabilistic methodology of \cite{F13} as a zero-diffusion limit of the latter.
The probabilistic approach in \cite{F13} computes singular vectors of an $\epsilon$-perturbed operator $\mathcal{L}_\epsilon$;  that is, eigenvectors of $\mathcal{L}_\epsilon^*\mathcal{L}_\epsilon$.
For $C^3$ $f:M\to\mathbb{R}$ we show that
\begin{equation*}
\label{diffform1thm_intro}
\lim_{\epsilon\to 0} \frac{(\mathcal{L}_\epsilon^*\mathcal{L}_\epsilon-I)f(x)}{\epsilon^2}=c\cdot(\triangle+\P^*\triangle\P)f(x),
\end{equation*}
for each $x\in \mathring{M}$, where $c$ is an explicit constant.
Thus, we provide a missing formal link between the probabilistic coherent set methodologies and direct notions of geometry via boundary size and volume.
Finally, in Section \ref{sect:numerics}, we illustrate how eigenfunctions of the dynamic Laplace operator (\ref{dynlap_intro}) can be used to find coherent sets using three numerical case studies.
The appendix contains most of the proofs.

\section{Background}
\label{staticsection}
Let $M$ be a compact, connected $d$-dimensional $C^\infty$ Riemannian manifold of vanishing curvature, which is either boundaryless or has $C^\infty$ boundary. 
This setting is relatively simple from a differential geometric point of view, but the introduction of nonlinear dynamics creates nontrivial questions in this setting.
Let $\ell_d$ denote Lebesgue (volume) measure on $M$.
To measure co-dimension 1 objects, we use $d-1$-dimensional Hausdorff measure $\mathcal{H}^{d-1}$ (using the trivial Riemannian metric to calculate diameter) to define $\ell_{d-1}=(\omega_{d-1}/2^{d-1})\mathcal{H}^{d-1}$;  see e.g.\ Corollary IV.1.1 \cite{chavelisoperimetric}.
We define the \emph{Cheeger constant}
\begin{equation}
\label{cheegerconst}\mathbf{h}:=\inf_\Gamma\frac{\ell_{d-1}(\Gamma)}{\min\{\ell_d(M_1),\ell_d(M_2)\}},
\end{equation}
where $\Gamma$ varies over compact $(d-1)$-dimensional $C^\infty$ submanifolds that separate $M$ into two connected components $M_1, M_2$.

One can link these geometric ideas with functions on $M$ by considering level sets of a function $f:M\to\mathbb{R}$ defining the $(d-1)$-dimensional separating surface $\Gamma$.
In fact, defining the \emph{Sobolev constant}
\begin{equation}
\label{sobolevconst}
\mathbf{s}:=\inf_{f\in C^\infty}\frac{\|\nabla f\|_1}{\inf_{\alpha\in\mathbb{R}}\|f-\alpha\|_1},
 \end{equation}
 one has the celebrated Federer-Fleming result
\begin{theorem}[\cite{federerfleming,mazya}]
\label{ffthm}
\begin{equation}
\label{staticffeqn}
\mathbf{s}=\mathbf{h}.
\end{equation}
\end{theorem}



A link to spectral theory of operators is provided by the Cheeger inequality.
Consider the eigenproblem
\begin{equation}
\label{closedeigenproblem}
\triangle\phi=\lambda\phi,\mbox{ on $\mathring{M}$}.
\end{equation}
If $\partial M\neq \emptyset$, then zero Neumann boundary conditions are imposed:
 \begin{equation}
\label{staticneumann}
\nabla \phi(y)\cdot \mathbf{n}(y)=0\mbox{ for $y\in\partial M$},
 \end{equation}
where $\mathbf{n}(y)$ is the outward unit normal to $\partial M$ at $y$.

It is well-known (see e.g.\ Theorem 1.1 \cite{chaveleigenvalues} and Section 4.4 \cite{mcowen}) that the set of eigenvalues consists of a sequence $0= \lambda_1>\lambda_2>\cdots\downarrow -\infty$.
and each associated eigenspace is finite-dimensional.
Eigenspaces belonging to distinct eigenvalues are orthogonal in $L^2(M)$ and $L^2(M)$ is the direct sum of all of the eigenspaces.
Furthermore, each eigenfunction is $C^\infty$ on $M$.

\begin{theorem}[Cheeger Inequality, \cite{cheeger}]
\label{cheegerthm}\
\begin{enumerate}
\item If $M$ is boundaryless, let $\lambda_2$ be the smallest magnitude nonzero eigenvalue of (\ref{closedeigenproblem}).
\item If $\partial M\neq\emptyset$, let $\lambda_2$ be the smallest magnitude nonzero eigenvalue for (\ref{closedeigenproblem})--(\ref{staticneumann}).
\end{enumerate}
Then
\begin{equation}
\label{staticcheegerineq}
\mathbf{h}\le 2\sqrt{-\lambda_2},
  \end{equation}
\end{theorem}

Level sets of the eigenfunction corresponding to $\lambda_2$ give vital information about the $\Gamma$ that achieves the Cheeger constant;  a fact that we will exploit in our new dynamic setup.
We remark that there is a vast literature on the use of the Laplace operator for extracting various types of geometric information on static manifolds and we refer the reader to the recent survey \cite{grebenkov_nguyen}, with over 500 references.
We now proceed through a few simple domains to illustrate the relationship between the solution to the isoperimetric problems and the eigenvalues and eigenfunctions of the Laplacian.

\subsection{The flat 2-torus:}
\label{sect:torus}
Consider the flat 2-torus $\mathbb{T}^2=2\pi(\mathbb{R}/\mathbb{Z})\times 2\pi(\mathbb{R}/\mathbb{Z})$, which we write as $[0,2\pi)/\sim\times [0,2\pi)/\sim$, where $\sim$ is the identification at the interval endpoints.
\subsubsection*{Solution to the isoperimetric problem}
There is an infinite family of optimal $\Gamma$ solving (\ref{cheegerconst}):  either $(\{x\}\times [0,2\pi))\cup(\{x+\pi\}\times [0,2\pi))$ (two vertical loops) or $([0,2\pi)\times \{y\})\cup([0,2\pi)\times\{y+\pi\})$ (two horizontal loops).
The value of $\mathbf{h}$ is $2\cdot(2\pi)/((1/2)\cdot (2\pi)^2)=2/\pi$.
One particular solution is shown as black lines in Figure \ref{torusfig}: $\Gamma=\{(\{\pi/2\}\times [0,2\pi))\cup(\{3\pi/2\}\times [0,2\pi))$, $M_1=[\pi/2,3\pi/2]\times [0,2\pi)$, and $M_2=\mathbb{T}^2\setminus M_1$.
\begin{figure}[hbt]
\centering
  \begin{subfigure}{0.3\textwidth}
  \includegraphics[width=\textwidth]{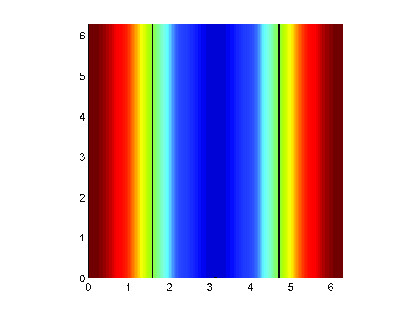}\\
  \caption{Torus $[0,2\pi)\times [0,2\pi)$: Plot of second Laplacian eigenfunction $\cos(x)$ with the the zero level set shown in black.}\label{torusfig}
  \end{subfigure}
  \begin{subfigure}{0.3\textwidth}
  \includegraphics[width=\textwidth]{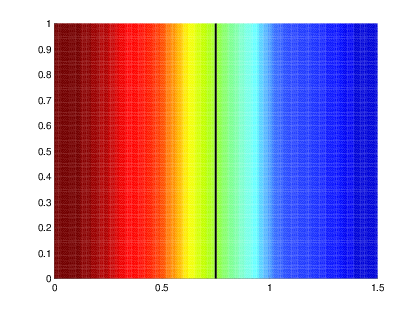}\\
  \caption{Rectangle $[0,a]\times [0,b]$, $a=3/2$, $b=1$: Plot of second Laplacian eigenfunction $\cos(4\pi x/9)$ with the the zero level set shown in black.}\label{rectfig}
  \end{subfigure}
  \begin{subfigure}{0.3\textwidth}
   \includegraphics[width=\textwidth]{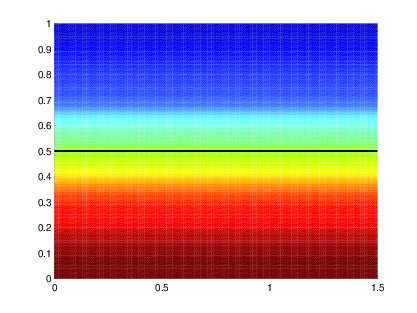}\\
  \caption{Cylinder $[0,a)\times [0,b]$, $a=3/2$, $b=1$: Plot of second Laplacian eigenfunction $\sin(\pi y)$ with the the zero level set shown in black.}\label{cylfig}
  \end{subfigure}
\end{figure}

\subsubsection*{Laplace operator and eigenfunctions}
\sloppy
The Laplace operator has eigenvalues $-(k^2+l^2)$, $k,l\in 0,1,2,\ldots,$ with eigenfunctions $\cos(k x)\cos(l y),\sin(k x)\cos(l y),\cos(kx)\sin(ly),\sin(kx)\sin(ly)$, $k,l\in 0,1,2,\ldots$. Thus the multiplicity of the first nontrivial eigenvalue -1 is 4, and the
corresponding eigenspace is spanned by $\{\cos(x),\sin(x),\cos(y),\sin(y)\}$.
The upper bound for $\mathbf{h}$ provided by Cheeger's inequality is 2.
Note that the zero level sets of the functions $\{\cos(x),\sin(x),\cos(y),\sin(y)\}$ are exactly the optimal disconnecting curves $\Gamma$ discussed above.
The zero level set of one of these functions, $\cos(x)$ is shown in black in Figure \ref{torusfig}.


%

\subsection{The rectangle}
\label{sect:rect}
Consider the rectangle $[0,a]\times [0,b]$, with $a>b$.
\subsubsection*{Solution to the isoperimetric problem}
The problem (\ref{cheegerconst}) has a unique solution $\Gamma$: $\{(a/2,y):0\le y\le b\}$, shown as a black line in Figure \ref{rectfig}.
The corresponding value of $\mathbf{h}$ is $2/a$.
\subsubsection*{Laplacian  operator and eigenfunctions}
The Laplace operator with zero Neumann boundary conditions has eigenvalues $-\pi^2(k^2/a^2+l^2/b^2)$,  $k,l\in 0,1,2,\ldots,$ with corresponding eigenfunctions $\cos(k\pi x/a)\cos(l\pi y/b)$, $k,l\in 0,1,2,\ldots$.
Note that the Neumann boundary conditions are satisfied on the boundary of the rectangle by this set of eigenfunctions.
The first nontrivial eigenvalue is $\lambda_2=-\pi^2/a^2$.
 In the example shown in Figure \ref{rectfig}, $a=3/2$, $b=1$ and the first three eigenvalues are $0,-4\pi^2/9,-\pi^2$, each of unit multiplicity, corresponding to $(k,l)=(0,0),(1,0),(0,1)$.
The corresponding eigenspaces are spanned by $1,\cos(3\pi x/2),\cos(\pi y)$.
Note that the zero level set of the second eigenfunction $\cos(3\pi x/2)$ is exactly the optimal disconnecting curve $\Gamma$, shown in black in Figure \ref{rectfig}.
The upper bound for $\mathbf{h}$ provided by Cheeger's inequality is $2\sqrt{4\pi^2/9}=4\pi/3$.


\subsection{The cylinder}
\label{sect:cyl}
Consider the flat cylinder $a(\mathbb{R}/\mathbb{Z})\times [0,b)$, which we write as $[0,a)/\sim\times [0,b]$, with $a>b$ and the vertical ``edges'' identified.
\subsubsection*{Solution to the isoperimetric problem}
The solution to the isoperimetric problem depends on the relative size of $a$ to $b$.
If $a<2b$ then $\mathbf{h}=a$, with a unique minimising disconnecting curve $\Gamma$: $\{(x,b/2):0\le x\le a\}$.
If $a>2b$ then $\mathbf{h}=2b$, with $\Gamma$ selected from an infinite family of pairs of vertical lines parameterised by $x\in[0,a)$: $\{(x,y):0\le y\le b\}\cup\{(x+a/2,y):0\le y\le b\}$.
\subsubsection*{Laplacian  operator and eigenfunctions}
The Laplace operator has eigenvalues $-\pi^2(4k^2/a^2+l^2/b^2)$, $k,l\in 0,1,2,\ldots,$ with corresponding eigenfunctions $\sin(2k\pi x/a)\cos(l\pi y/b),\cos(2k\pi x/a)\cos(l\pi y/b)$, $k,l\in 0,1,2,\ldots$.
Note that we only have to enforce Neumann boundary conditions on the top and bottom horizontal boundaries of the cylinder.
The leading eigenvalue is  0, and the second eigenvalue depends on the relative size of $a$ and $b$; 
a switch occurs at $a=2b$, matching the corresponding switch in the domain geometry.
 In the example shown in Figure \ref{cylfig}, $a=3/2$, $b=1$ and the first three eigenvalues are $0,-\pi^2,-16\pi^2/9$, with multiplicities 1, 1, and 2, corresponding to $(k,l)=(0,0),(0,1),(1,0)$.
The corresponding eigenspaces are spanned by $1,\cos(\pi y),\{\sin(4\pi x/3),\cos(4\pi x/3)\}$.
The upper bound for $\mathbf{h}$ provided by Cheeger's inequality is $2\pi$.
Note that the zero level set of the second eigenfunction $ \cos(\pi y)$ is exactly the optimal disconnecting curve $\Gamma$, shown in Figure \ref{cylfig}.

\section{Dynamic Isoperimetry}
In this section we extend the concepts of the previous section to a dynamic setting, where $T$ is a $C^\infty$ diffeomorphism from $M$ onto $T(M)$, and $M$, $T(M)$ are compact, connected Riemannian manifolds of vanishing curvature.
Much of the discussion in this paper is for a single iterate of a $T:M\to T(M)$, however, the extension to multiple iterates of the same map, iterates of different maps as would occur in time-dependent dynamical systems, and even a continuum of flow maps generated by a time-dependent ODE is straightforward (see Section \ref{sect:multistep}).
For a single iterate of $T$ we seek sets that have small boundary size relative to volume both \emph{before} and \emph{after} the application of the nonlinear dynamics of $T$.
Thus, if $\Gamma$ is the boundary of a coherent set, one needs to minimise both $\ell_{d-1}(\Gamma)$ and $\ell_{d-1}(T\Gamma)$.
To identify finite-time coherent sets, we propose the following natural dynamic minimisation problem.
\begin{definition}
\label{cheegerdefn}
Define the \emph{dynamic Cheeger constant} $\mathbf{h}^D$ by
\begin{equation}
\label{cheegereqn2}
\mathbf{h}^D:=\inf_\Gamma \frac{\ell_{d-1}(\Gamma)+\ell_{d-1}(T(\Gamma))}{2\min\{\ell(M_1),\ell(M_2)\}}.
\end{equation}
where $\Gamma$ varies over compact $(d-1)$-dimensional $C^\infty$ submanifolds of $M$ that divide $M$ into two disjoint open submanifolds $M_1, M_2$ of $M$.
\end{definition}

In the present paper, to avoid obscuring the key constructions, we focus on volume-preserving $T$.
Note that (\ref{cheegereqn2}) cannot be decomposed into two static minimisation problems because $\Gamma$ is the same in both terms in the numerator of (\ref{cheegereqn2}).

\subsection{A Dynamic Federer-Fleming Theorem}
\label{sect:ff}

It is of theoretical interest (and for the present paper, of interest in dynamical systems applications) to connect the set-based optimisation problem (\ref{cheegereqn2}) with functional optimisation problems.
Two basic tools in differential geometry for doing this are the co-area formula, which connects spatial integrals of the gradient of a function with an integral over the areas of level sets of a function, and Cavalieri's principle, which represents a function as an integral over its level sets.
The Federer-Fleming theorem (Theorem \ref{ffthm}) connects a set-based isoperimetric problem (the (static) Cheeger constant) in an exact way with a functional minimisation problem.
We wish to formulate a dynamic equivalent of this theorem.

Let $M$ be a compact, connected Riemannian manifold of dimension $d\ge 1$ with vanishing curvature, and $T:M\to T(M)$ a volume-preserving diffeomorphism.
We denote by $\mathcal{P}$ the Perron-Frobenius operator of $T$, defined by $\mathcal{P}f=f\circ T^{-1}$ as $T$ is volume-preserving.
\begin{definition}
\label{sobolevdefn}
Define the \emph{dynamic Sobolev constant of $M$}, $\mathbf{s}^D(M)$ by
\begin{equation}
\label{soboleveqn}
\mathbf{s}^D=\inf_{f\in C^\infty} \frac{\|\nabla f\|_1+\|\nabla(\P f)\|_1}{2\inf_{\alpha\in \mathbb{R}} \|f-\alpha\|_{1}}.
\end{equation}
\end{definition}
Related to the above is the alternate Sobolev constant
\begin{equation}
\label{altsoboleveqn}
\hat{\mathbf{s}}^D=\inf_{f\in C^\infty} \frac{\|\nabla f\|_1+\|\nabla(\P f)\|_1}{2\|f-\bar{f}\|_{1}},
\end{equation}
setting $\alpha$ to be the mean value of $f$, $\bar{f}=(1/\ell(M))(\int_M f\ d\ell)$ (see e.g.\ \cite{chavelisoperimetric} p163 for the static version).
Clearly, $\hat{\mathbf{s}}^D\le \mathbf{s}^D$.
We wish to demonstrate a dynamic analogue of the Federer-Fleming theorem.
Our first main result is:
\begin{theorem}[Dynamic Federer-Fleming Theorem]
\label{fflemma}
Let $M$ be a compact, connected $C^\infty$ manifold with vanishing curvature.
Let $T:M\to T(M)$ be a $C^\infty$ volume-preserving diffeomorphism. Then
\begin{equation}
\label{ffeqn}
\mathbf{s}^D=\mathbf{h}^D,
\end{equation}
and further,
\begin{equation}
\label{ffeqn22}
\mathbf{h}^D/2\le \hat{\mathbf{s}}^D\le \mathbf{s}^D=\mathbf{h}^D.
\end{equation}
\end{theorem}
\begin{proof}
See appendix.
\end{proof}


\subsection{A Dynamic Cheeger Inequality}
\label{sect:cheegerineq}
The (static) Cheeger inequality (Theorem \ref{cheegerthm}) is an $L^2$-based result while the (static) Federer-Fleming equality (Theorem \ref{ffthm}) is $L^1$-based.
The advantage of $L^2$ is that one obtains a nice spectral theory for $\triangle$ from the Hilbert space structure, and crucial variational characterisations of the eigenvalues.
One pays for this convenience by obtaining an inequality, rather than equality.
Nevertheless, as we have seen in Sections \ref{sect:torus}--\ref{sect:cyl}, the level sets of the Laplacian eigenfunctions carry significant information and provide good solutions to the original set-based isoperimetric problem (\ref{cheegerconst}).
We wish to replicate these properties for the dynamic Cheeger constant $\mathbf{h}^D$ and a dynamic version of the Laplace operator.
We define the latter by
\begin{equation}
\label{hattriangle}
\hat{\triangle}:=(\triangle+\mathcal{P}^*\triangle\mathcal{P})/2.
\end{equation}
The spectral properties of this operator are developed in Section \ref{sect:dynlapspec}, but we say a few words here about the intuition behind this definition.
Consider a function $f:M\to \mathbb{R}$ on $M$ at the initial time from which we extract level sets, as in Figures (\ref{torusfig})--(\ref{cylfig}).
The first term in (\ref{hattriangle}), $\triangle$, is the Laplacian on the domain $M$ and of obvious importance for providing information on decompositions of $M$.
The second term $\mathcal{P}^*\triangle\mathcal{P}$ first pushes the function on $f$ on $M$ forward to a function $\P f$ on $T(M)$, possibly undergoing nonlinear distortion.
One then applies the Laplacian to $\P f$ on $T(M)$ to
 obtain geometric information on $T(M)$, and finally pulls the result back to $M$ with $\P^*$, ready to be combined with the result from the first term $\triangle$.
We note that in fact $\mathcal{P}^*\triangle\mathcal{P}$ is the Laplace-Beltrami operator for the pullback of the Euclidean metric on $T(M)$.
Consider $\triangle_\delta:C^\infty(T(M),\mathbb{R})\circlearrowleft$ as the Laplace-Beltrami operator on the Riemannian manifold $(T(M),\delta)$, where $\delta$ denotes the Riemannian metric (in the present context, $\delta$ is the trivial Euclidean metric).
Pulling $\delta$ back under $T$ we obtain the Riemannian metric $T^*\delta$ and the map $T:(M,T^*\delta)\to (T(M),\delta)$ is an isometry.
One can now write $\triangle_{T^*\delta}f=(\triangle_\delta(f\circ T^{-1}))\circ T=\mathcal{P}^*\triangle_\delta\mathcal{P}$; see e.g.\ p27 \cite{chaveleigenvalues}.

Our second new result is a dynamic Cheeger inequality, which highlights the importance of eigenfunctions of the operator $\hat{\triangle}$.

\begin{theorem}
\label{cheegerdynthm}
Let $M$ be a compact, connected $C^\infty$ manifold with vanishing curvature, and $T:M\to T(M)$ be a $C^\infty$ and volume-preserving diffeomorphism.
\begin{enumerate}
\item If $M$ is boundaryless, then let $\lambda_2$ be the smallest magnitude nonzero eigenvalue of $\hat{\triangle}$.
\item If $\partial M\neq\emptyset$, denote by $\mathbf{n}(x)$ the outward unit normal at $x\in \partial M$.
Let $\lambda_2$ be the smallest magnitude nonzero eigenvalue for the $L^2$-eigenproblem
\begin{equation}
\label{strongeqn0}
\hat{\triangle}u(x)=\lambda u(x),\quad x\in \mathring{M},
\end{equation}
with boundary condition
\begin{equation}
\label{strongbc0}
\nabla u(x)\cdot\left[\left(I+DT(x)^{-1}\left(DT(x)^{-1}\right)^\top\right) \mathbf{n}(x)\right]=0,\quad x\in\partial M.
\end{equation}
\end{enumerate}
Then
\begin{equation}
\label{dyncheegereqn}
\mathbf{h}^{D}\le 2\sqrt{-\lambda_2}.
\end{equation}
\end{theorem}
\begin{proof}
See appendix.
\end{proof}

An intuitive explanation of the term
$\left((\nabla u(x))^\top DT(x)^{-1}\right)\cdot\left(\left(DT(x)^{-1}\right)^\top\mathbf{n}(x)\right)$
 in (\ref{strongbc0}) is that (i) $(\nabla u(x))^\top DT(x)^{-1}=\nabla(u\circ T^{-1})(T(x))$ (gradient of the pushforward of $u$ by $T$ at $T(x)$) and (ii) $\left(DT(x)^{-1}\right)^\top\mathbf{n}(x)$ is normal to $\partial T(M)$ at $T(x)$.
Thus, $\left((\nabla u(x))^\top DT(x)^{-1}\right)\cdot\left(\left(DT(x)^{-1}\right)^\top\mathbf{n}(x)\right)=0$ can be viewed as a natural pullback of a zero Neumann boundary condition on $\partial T(M)$ at $T(x)$.
In terms of metrics, one has $\nabla_\delta(u\circ T^{-1})=(\nabla_{T^*\delta}(u))\circ T^{-1}$.
\begin{remark}
\label{proofremark}
Using the above pullback interpretation of the boundary condition and the pullback interpretation of $\hat{\triangle}$, one can produce shorter, coordinate-free proofs of Theorems \ref{cheegerdynthm} and \ref{dynlapthm}, instead of the coordinate-based proofs in the Appendix.
Similarly, the proof of Theorem \ref{fflemma} can also be easily approached from this point of view.
\end{remark}

\subsection{Multiple time-steps}
\label{sect:multistep}
Let us now consider a composition of several maps $T_1,\ldots,T_{n-1}$, denoting $T^{(i)}:=T_i\circ\cdots\circ T_2\circ T_1$, $i=1,\ldots,n-1$.
These maps might arise, for example, as time-$\tau$ maps of a time-dependent flow.
If we wish to track the evolution of a coherent set under these maps, penalising the boundary of the evolved set $T^{(i)}(\Gamma)$ after the application of each $T_i$, then we can define
\begin{equation}
\label{hdynn}
\mathbf{h}_n^{D}:=\inf_\Gamma \frac{\frac{1}{n}\sum_{i=0}^{n-1}\ell_{d-1}(T^{(i)}\Gamma)}{\min\{\ell(M_1),\ell(M_2)\}},
\end{equation}
as the natural generalisation of $\mathbf{h}^D$.

In continuous time, we consider a (possibly time-dependent) ODE $\dot{x}=F(x,t)$, where $F$ is $C^\infty$ on $M\times[0,\tau]$.
The flow maps $T^{(t)}:M\to T^{(t)}(M)$ are then smooth\footnote{To weaken the smoothness assumption on $F(x,\cdot)$, but still obtain smooth flow maps, see \cite{arnold} Appendix B.3.} for each $t\in[0,\tau]$.
One can define
\begin{equation}
\label{hdynt}
\mathbf{h}_{[0,\tau]}^{D}:=\inf_\Gamma \frac{\frac{1}{\tau}\int_0^\tau \ell_{d-1}(T^{(t)}\Gamma)\ dt}{\min\{\ell(M_1),\ell(M_2)\}},
\end{equation}
as a time-continuous generalisation of $\mathbf{h}^D$.

Analogously, setting $\mathcal{P}^{(i)}f=f\circ (T^{(i)})^{-1}$ and $\mathcal{P}^{(t)}f=f\circ (T^{(t)})^{-1}$, one can define dynamic Sobolev constants for multiple discrete time steps or over a continuous time interval:
\begin{equation}
\label{sdynnt}
\mathbf{s}^D_n=\inf_{f\in C^\infty} \frac{\frac{1}{n}\sum_{i=0}^{n-1}\|\nabla ( \mathcal{P}^{(i)}f)\|_1}{\inf_{\alpha\in \mathbb{R}} \|f-\alpha\|_{1}},\qquad\qquad \mathbf{s}^D_{[0,\tau]}=\inf_{f\in C^\infty} \frac{\frac{1}{\tau}\int_0^\tau\|\nabla ( \mathcal{P}^{(t)}f)\|_1\ dt}{\inf_{\alpha\in \mathbb{R}} \|f-\alpha\|_{1}}.
\end{equation}

\begin{corollary}[Multistep Dynamic Federer-Fleming Theorem]
\label{ffcor}
Let $M$ be a compact, connected $C^\infty$ manifold with vanishing curvature, and $T^{(i)}$, $i=1,\ldots,n-1$ (resp.\ $T^{(t)}, t\in[0,\tau]$) be generated by a sequence of $C^\infty$ volume-preserving diffeomorphisms (resp.\ be smooth flow maps generated by a volume-preserving ODE $\dot{x}=F(x,t)$).
Then
\begin{equation}
\label{ffeqnn}
\mathbf{h}^D_n/2\le \hat{\mathbf{s}}^D_n\le \mathbf{s}^D_n=\mathbf{h}^D_n,
\end{equation}
resp.\
\begin{equation}
\label{ffeqnt}
\mathbf{h}^D_{[0,\tau]}/2\le \hat{\mathbf{s}}^D_{[0,\tau]}\le \mathbf{s}^D_{[0,\tau]}=\mathbf{h}^D_{[0,\tau]}.
\end{equation}
\end{corollary}
\begin{proof}
See appendix.
\end{proof}

Theorem \ref{cheegerdynthm} also naturally extends to multiple time steps.
\begin{corollary}[Multistep Cheeger Inequality -- discrete time]
\label{cheegerncor}
Let $M$ be a compact, connected $C^\infty$ manifold with vanishing curvature, and $T^{(i)}$, $i=1,\ldots,n-1$ be generated by a sequence of $C^\infty$ volume-preserving diffeomorphisms.
Define
\begin{equation}
\label{bartrianglen}
\hat{\triangle}^{(n)}:=\frac{1}{n}\sum_{i=0}^{n-1}(\mathcal{P}^{(i)})^*\triangle\mathcal{P}^{(i)},
\end{equation}
where $\mathcal{P}^{(i)}f=f\circ (T^{(i)})^{-1}$.
\begin{enumerate}
\item If $\partial M=\emptyset$, let $\lambda_2$ be the smallest magnitude nonzero eigenvalue of $\hat{\triangle}^{(n)}$.
\item If $\partial M\neq\emptyset$, denote by $\mathbf{n}(x)$ the outward unit normal at $x\in \partial M$.
Let $\lambda_2$ be the smallest magnitude nonzero eigenvalue for the $L^2$ eigenproblem
\begin{equation}
\label{strongeqnmultn}
\hat{\triangle}^{(n)}u(x)=\lambda u(x),\quad x\in \mathring{M},
\end{equation}
with boundary condition
\begin{equation}
\label{strongbcmultn}
\nabla u(x)\cdot\left[\sum_{i=0}^{n-1}DT^{(i)}(x)^{-1}\left(DT^{(i)}(x)^{-1}\right)^\top \mathbf{n}(x)\right]=0,\quad x\in\partial M.
\end{equation}
\end{enumerate}
Then
\begin{equation}
\label{cheegermultistepn}
\mathbf{h}_n^{D}\le 2\sqrt{-\lambda_2}.
\end{equation}
\end{corollary}
\begin{proof}
See Appendix.
\end{proof}
\begin{corollary}[Multistep Cheeger Inequality -- continuous time]
\label{cheegertcor}
Let $M$ be a compact, connected $C^\infty$ manifold with vanishing curvature, and $T^{(t)}, t\in[0,\tau]$ be smooth flow maps.
Define
\begin{equation}
\label{bartrianglet}
\hat{\triangle}^{(\tau)}:=\frac{1}{\tau}\int_0^\tau(\mathcal{P}^{(t)})^*\triangle\mathcal{P}^{(t)}\ dt,
\end{equation}
where $\mathcal{P}^{(t)}f=f\circ (T^{(t)})^{-1}$.
\begin{enumerate}
\item If $\partial M=\emptyset$, let $\lambda_2$ be the smallest magnitude nonzero eigenvalue of $\hat{\triangle}^{(\tau)}$.
\item If $\partial M\neq\emptyset$, denote by $\mathbf{n}(x)$ the outward unit normal at $x\in \partial M$.
Let $\lambda_2$ be the smallest magnitude nonzero eigenvalue for the $L^2$ eigenproblem
\begin{equation}
\label{strongeqnmultt}
\hat{\triangle}^{(\tau)}u(x)=\lambda u(x),\quad x\in \mathring{M},
\end{equation}
with boundary condition
\begin{equation}
\label{strongbcmultt}
\nabla u(x)\cdot\left[\int_0^\tau DT^{(t)}(x)^{-1}\left(DT^{(t)}(x)^{-1}\right)^\top \mathbf{n}(x)\ dt\right]=0,\quad x\in\partial M.
\end{equation}
\end{enumerate}
Then
\begin{equation}
\label{cheegermultistept}
\mathbf{h}_\tau^{D}\le 2\sqrt{-\lambda_2}.
\end{equation}
\end{corollary}
\begin{proof}
See Appendix.
\end{proof}
\begin{remark}
If one does not wish to track the length of the evolved $\Gamma$ except at the initial and final time, one would instead use (\ref{hattriangle}) with $\mathcal{P}=\mathcal{P}^{(n-1)}$ or $\mathcal{P}=\mathcal{P}^{(\tau)}$.
\end{remark}

\section{Spectral properties of the dynamic Laplacian}
\label{sect:dynlapspec}
The following result summarises important properties of the operator $\hat{\triangle}=(\triangle+\P^*\triangle \mathcal{P})/2$.

\begin{theorem}
\label{dynlapthm}
Let $M$ be a compact, connected $C^\infty$ manifold with vanishing curvature, and $T:M\to T(M)$ be a $C^\infty$, volume-preserving diffeomorphism.
\begin{itemize}
\item If $M$ is boundaryless, let $\lambda, u$ denote solutions to the $L^2$ eigenproblem $(1/2)(\triangle+\mathcal{P}^*\triangle\mathcal{P})u=\lambda u$ on $M$.
\item If $\partial M\neq\emptyset$, denote by $\mathbf{n}(x)$ the outward unit normal at $x\in \partial M$.
Let $\lambda, u$ denote solutions to the $L^2$-eigenproblem
\begin{equation}
\label{strongeqn}
(1/2)(\triangle+\P^*\triangle\P)u(x)=\lambda u(x),\quad x\in \mathring{M},
\end{equation}
with boundary condition
\begin{equation}
\label{strongbc}
\nabla u(x)\cdot\left[\mathbf{n}(x)+DT(x)^{-1}\left(DT(x)^{-1}\right)^\top \mathbf{n}(x)\right]=0,\quad x\in\partial M.
\end{equation}
\end{itemize}
The solutions $\lambda, u$ satisfy the following properties.

\begin{enumerate}
\item The eigenvalues form a decreasing sequence $0=\lambda_1> \lambda_2>\cdots$ with $\lambda_n\to-\infty$.
\item The corresponding eigenfunctions $u_1,u_2,\ldots$ are $C^\infty$ on $M$ and eigenfunctions corresponding to distinct eigenvalues are pairwise orthogonal in $L^2$.
\item One has the variational characterisation of eigenvalues: if $u_1,u_2,\ldots$ are arranged to be orthonormal, denoting $X_k=\span\{u_1,u_2,\ldots,u_k\}$
\begin{equation}
\label{variational}
\lambda_k=-\inf_{u\in X,\langle u,u_i\rangle=0, i=1,\ldots,k-1}\frac{\int_M |\nabla u|^2\ d\ell+\int_{T(M)}|\nabla(\P u)|^2\ d\ell}{2\int_M u^2\ d\ell},
\end{equation}
with the infimum achieved only when $u=u_k$.
\end{enumerate}
\end{theorem}
Our main focus is the eigenvalue $\lambda_2$ and the corresponding eigenfunction $u_2$.
We will see that $u_1\equiv 1$ and therefore that $\int_M u_2\ d\ell=0$ by $L^2$-orthogonality to $u_1$.
Appendix \ref{sec:weakexist} contains the proofs of items 1 and 3 of Theorem \ref{dynlapthm} and Appendix \ref{sec:ellipticity} contains the proofs of item 2 and the boundary conditions.

\begin{remark}
\label{spectrummultistep}
Using identical arguments, one can also show a multiple time step version of Theorem \ref{dynlapthm}, where
$(1/2)(\triangle+\P^*\triangle\P)$ is replaced with either (\ref{bartrianglen}) or (\ref{bartrianglet}), and the boundary condition (\ref{strongbc}) is replaced with either (\ref{strongbcmultn}) or (\ref{strongbcmultt}).
\end{remark}

\subsection{Objectivity}

We demonstrate that the operator $\hat{\triangle}$ behaves in a very predictable way when the phase space is observed in a time-dependent rotating and translating frame.
In particular, we show that the method of extracting coherent sets from eigenvectors of $\hat{\triangle}$ (described in Section \ref{sect:numerics}) is \emph{objective} or \emph{frame-invariant}, meaning that the method produces the same features when subjected to time-dependent ``proper orthogonal + translational'' transformations; see \cite{truesdellnoll}.

In continuous time, to test for objectivity, one makes a time-dependent coordinate change $x\mapsto Q(t)x+b(t)$ where $Q(t)$ is a proper othogonal linear transformation and $b(t)$ is a translation vector, for $t\in[t_0,t_1]$.
The discrete time analogue is to imagine we begin in the frame given by $\Phi_{t_0}(M)$,
where $\Phi_{t_0}(x)=Q(t_0)x+b(t_0)$, and end in the frame $\Phi_{t_1}(M)$, where $\Phi_{t_1}(x)=Q(t_1)x+b(t_1)$.
We are concerned with the deterministic transformation $\dot{T}:\Phi_{t_0}(M)\to \Phi_{t_1}(M)$, which is given by $\dot{T}=\Phi_{t_1}\circ T\circ \Phi_{t_0}^{-1}$.
This change of frames is summarised in the commutative diagram below.
$$
\begin{CD}
M @>T>> T(M)\\
@VV\Phi_{t_0}V @VV\Phi_{t_1}V\\
\Phi_{t_0}(M) @>\dot{T}>> \Phi_{t_0}\circ T(M)
\end{CD}
$$
Corresponding to $\dot{T}$ is the operator $\dot{\hat{\triangle}}=\triangle+\P^*_{\dot{T}}\triangle \P_{\dot{T}}$, where $\P_{\dot{T}}=\P_{\Phi_{t_1}}\circ \P\circ \P_{\Phi_{t_0}}^{-1}$.
If we were observing the dynamics in the frames given by $\Phi_{t_0}$ and $\Phi_{t_1}$ we would compute eigenfunctions of $\dot{\hat{\triangle}}$.
\begin{theorem}
\label{objthm}
The operator $\hat{\triangle}$ in the original frame and the operator $\dot{\hat{\triangle}}$ in the transformed frame satisfy the commutative diagram:
$$
\begin{CD}
L^2(M) @>\hat{\triangle}>> L^2(M)\\
@VV\P_{\Phi_{t_0}}V @VV\P_{\Phi_{t_0}}V\\
L^2(\Phi_{t_0}(M))@>\dot{\hat{\triangle}}>>L^2(\Phi_{t_1}(M))
\end{CD}
$$
Consequently, if $f$ solves
\begin{equation}
\label{origevalprob}
\hat{\triangle}f=\lambda f\quad\mbox{ on $\mathring{M}$},
\end{equation}
then
\begin{equation}
\label{transfevalprob}
\dot{\hat{\triangle}}(\P_{\Phi_{t_0}}f)=\lambda (\P_{\Phi_{t_0}}f)\quad\mbox{ on $\Phi_{t_0}(\mathring{M})$}.
\end{equation}
Furthermore, if $\partial M\neq \emptyset$, then if
\begin{equation}
\label{strongorigbc}
\nabla f(x)\cdot\left[\mathbf{n}(x)+DT(x)^{-1}\left(DT(x)^{-1}\right)^\top\mathbf{n}(x)\right]=0,\quad x\in\partial M,
\end{equation}
one has
\begin{equation}
\label{strongtransfbc}
\nabla (\P_{\Phi_{t_0}}f)(x)\cdot\left[\dot{\mathbf{n}}(x)+D\dot{T}(x)^{-1}\left(D\dot{T}(x)^{-1}\right)^\top\dot{\mathbf{n}}(x)\right]=0,\quad x\in\partial(\Phi_{t_0}(M)),
\end{equation}
where $\dot{\mathbf{n}}(x)=Q(t_0)\mathbf{n}(\Phi_{t_0}^{-1}x)$.
%
%
\end{theorem}
It follows from Theorem \ref{objthm} that the coherent sets extracted on $M$ from e.g.\ level sets of the eigenfunctions of $\dot{\hat{\triangle}}$ will be transformed versions (under $\Phi_{t_0}$) of those extracted from $\hat{\triangle}$, as required for objectivity.

\section{Zero-diffusion limit of an analytic diffusion-based framework}
\label{sect:zerolimit}

The paper \cite{F13} introduced an analytic methodology for finding finite-time coherent sets, formalising prior numerical work \cite{FSM10}.
This methodology was based around smoothings of $\P$.
In \cite{F13}, one defined smoothing operators $\mathcal{D}_{M,\epsilon}:L^2(M)\to \mathcal{H}_{1/2}(M_\epsilon)$, $\mathcal{D}_{T(M_\epsilon),\epsilon}:L^2(T(M_\epsilon))\to \mathcal{H}_{1/2}(T(M)_\epsilon)$ where  $\mathcal{H}_{1/2}(M_\epsilon), \mathcal{H}_{1/2}(T(M)_\epsilon)$ denote H\"{o}lder functions with exponent 1/2 on $\epsilon$-neighbourhoods of $M$ and $T(M_\epsilon)$, respectively.
The operators considered in \cite{F13} were
$\mathcal{D}_{M,\epsilon} f(y)=\int_M \alpha_\epsilon(x-y)f(x)\ d\ell(x)$, $y\in M_\epsilon$ and $\mathcal{D}_{T(M_\epsilon),\epsilon} f(y)=\int_{T(M_\epsilon)} \alpha_\epsilon(x-y)f(x)\ d\ell(x)$, $y\in T(M)_\epsilon$, where $\alpha_\epsilon(x)=\mathbf{1}_{B_\epsilon(0)}/\ell(B_\epsilon(0))$, corresponds to smoothing on a local $\epsilon$-ball.
In \cite{F13}, the operator $\mathcal{L}_\epsilon:L^2(M)\to L^2(T(M)_\epsilon)$, $\mathcal{L}_\epsilon=\mathcal{D}_{T(M_\epsilon),\epsilon}\P\mathcal{D}_{M,\epsilon}$ was introduced\footnote{in \cite{F13} there is an additional normalisation term required for non-Lebesgue reference measures and non-Lebesgue-preserving $T$.  In the present paper, as $T$ is volume preserving and volume is our reference measure, we eschew this normalisation term here.} and used to identify coherent sets for $T$ in phase space $M$.
The reason for the diffusion operators $\mathcal{D}_\epsilon$ were two-fold.
\begin{enumerate}
\item
Firstly, as $T$ is often invertible (e.g.\ the time-$t$ map of some smooth flow), subsets of $M$ are simply deformed by $T$, they do not ``disperse'', and one could argue that every set is ``coherent'' in the sense that it is non-dispersive.
Let us consider the action of $\mathcal{L}_\epsilon$ on $\mathbf{1}_{A}$, where the latter represents a subset $A\subset M$ by its characteristic function;  we think of $\mathbf{1}_{A}$ as a uniform mass distribution on $A$.
Applying $\mathcal{L}_\epsilon$, we first have $\mathcal{D}_{M,\epsilon}$ acting on $\mathbf{1}_{A}$, which removes from $A$ some mass within distance $\epsilon$ of the boundary of $A$.
The resulting function is then transformed dynamically by $\P$, and will be supported on an $\epsilon$-neighbourhood of $T(A)$.
Finally, we apply $\mathcal{D}_{T(M_\epsilon)}$ again, so that some mass within a distance $\epsilon$ of the boundary of the support of $\P\mathcal{D}_{M,\epsilon}\mathbf{1}_A$ is ejected from this support.
These ideas are quantified in the proof of Lemma 6 \cite{F13}.
In this way, the boundary size of both $A$ and $T(A)$ are penalised because the amount of mass ejected by the operators $\mathcal{D}_{M,\epsilon},\mathcal{D}_{T(M_\epsilon),\epsilon}$ is proportional to the boundary sizes.
\item Secondly, in order to find a set $A$ with minimal combined boundary sizes for $A$ and $T(A)$, \cite{F13} used minimisation properties of the singular vectors of $\mathcal{L}_\epsilon$;  in particular, the sets $A$ and $T(A)$ were estimated from the left/right singular vectors corresponding to the second largest singular value of $\mathcal{L}_\epsilon$ (the leading singular value is always 1 by construction).
    To use this variational machinery $\mathcal{L}_\epsilon$ needs to be compact, and it was shown in \cite{F13} that $\mathcal{D}_{M,\epsilon}, \mathcal{D}_{T(M_\epsilon)}$ also played the technical role of ensuring compactness of $\mathcal{L}_\epsilon$ acting on $L^2$ functions.
\end{enumerate}
The singular vector of $\mathcal{L}_\epsilon$ that corresponds to the initial time (prior to application of $T$) is an eigenvector of $\mathcal{A}_\epsilon:=\mathcal{L}_\epsilon^*\mathcal{L}_\epsilon$;  one pushes forward and then pulls back.
Without any diffusion operators, this would read $\mathcal{A}_0=\P^*\P$;  deterministically pushing forward and deterministically pulling back.
Because $T$ is volume-preserving and invertible, $\P f=f\circ T^{-1}$ and $\P^*=f\circ T$.
Thus $\mathcal{A}_0$ is the identity operator, and one lacks compactness and a ``second'' eigenvalue.
Without diffusion, there is no distinguished coherent set, all sets are equally coherent as they are merely distorted, not dispersed, by the deterministic dynamics over the finite time duration encoded in $T$.

However, one can ask about higher order terms when $\epsilon$ is  close to zero.
We show that with the right scaling in $\epsilon$, one can make sense of an expression like
\begin{equation}
\label{difflimit}
\mathcal{B}f(x):=\lim_{\epsilon\to 0}\frac{(\mathcal{L}_\epsilon^*\mathcal{L}_\epsilon-I)f(x)}{\epsilon^\beta},
\end{equation}
with $\mathcal{B}$ capturing the essential effects of tiny $\epsilon$-diffusion \emph{without  explicitly including} that diffusion.

We slightly modify and generalise the diffusion operators from \cite{F13}.
Let $q:M\to\mathbb{R}^+$ be a nonnegative density with compact support, with mean the origin, and with covariance matrix $c\cdot I$, where $I$ is the $d\times d$ identity matrix.
We scale $q$ to form $q_\epsilon(x)=q(x/\epsilon)/\epsilon^d$;  $q_\epsilon$ will play the role of the previous $\alpha_\epsilon$, and obviously $q(x)=\mathbf{1}_{B_1(0)}/\ell(B_1(0))$ is one example of a density satisfying the above conditions.
We redefine $\mathcal{D}_{M,\epsilon}f(x)=\int_M q_\epsilon(x-y)f(y)\ d\ell(y)$, $x\in \mathring{M}$, and $\mathcal{D}_{T(M),\epsilon}f(x)=\int_{T(M)} q_\epsilon(x-y)f(y)\ d\ell(y)$, $x\in \mathring{T(M)}$,
where $\epsilon=\epsilon(x)$ is sufficiently small that both operators preserve integrals (i.e.\ $\int_M f\ d\ell=\int_M \mathcal{D}_{M,\epsilon}f\ d\ell$ and $\int_{T(M)} f\ d\ell=\int_{T(M)} \mathcal{D}_{T(M),\epsilon}f\ d\ell$).
In the sequel we use the definition $\mathcal{L}_\epsilon=\mathcal{D}_{T(M),\epsilon}\P\mathcal{D}_{M,\epsilon}$.
With the additional assumptions $\int_M q_\epsilon(x-y)^2\ d\ell(y)d\ell(x), \int_{T(M)} q_\epsilon(x-y)^2\ d\ell(y)d\ell(x)<\infty$, one has $\mathcal{L}_\epsilon:L^2(M)\to L^2(T(M))$ is compact as required in \cite{F13}.
The following theorem shows that one can in fact take the scaling limit (\ref{difflimit}) with $\beta=2$ and that $\mathcal{B}$ is a scalar multiple (the variance of the diffusion) of $\hat{\triangle}$.

\begin{theorem}
\label{analcvgce}
Let $M$ be a connected, compact Riemannian manifold of vanishing curvature, $f:M \to\mathbb{R}$ be $C^3$, and $T:M\to T(M)$ be $C^3$ and volume-preserving.
Let $q:M\to\mathbb{R}^+$ be a nonnegative density with compact support, with mean the origin, and covariance matrix $c\cdot I$, where $I$ is the $d\times d$ identity matrix, and let $\mathcal{L}_\epsilon=\mathcal{D}_{T(M),\epsilon}\P\mathcal{D}_{M,\epsilon}$ be defined as above.
One has
\begin{equation}
\label{diffform1thm}\lim_{\epsilon\to 0} \frac{(\mathcal{L}_\epsilon^*\mathcal{L}_\epsilon-I)f(x)}{\epsilon^2}=c\cdot(\triangle+\P^*\triangle\P)f(x),
\end{equation}
for each $x\in \mathring{M}$.
\end{theorem}
The proof of Theorem \ref{analcvgce} is in Appendix \ref{sec:analcvgce}.
The appearance of the Laplace operator is due to the fact that $\mathcal{D}_{M,\epsilon}f(x)\approx f(x)+ (c\epsilon^2/2)\triangle f(x)$ for small $\epsilon$.
The symmetry conditions on $q$ in Theorem \ref{analcvgce} (which are physically desirable as they model isotropic diffusion) cause the first order term in $\epsilon$ to vanish.
\begin{example}
\label{unifexample}
If $q(x)=\mathbf{1}_{B_1(0)}/\ell(B_1(0))$ (uniform diffusion on a unit ball), then $c=1/3, 1/4, 1/5$ in dimensions $d=1, 2, 3$, respectively.
\end{example}

Theorem \ref{analcvgce} provides a theoretical link between the diffusion-based method \cite{F13} and the diffusion-free constructions based on the Laplace operator in the present paper.
The latter have very strong connections with geometry, evidenced by Theorems \ref{fflemma} and \ref{cheegerdynthm}, and further reinforce the geometric intuition of \cite{F13}.
In numerical computations, if the dynamical system is deterministic and the dynamics and the domain are smooth, the present construction may be advantageous because the spectrum of $\hat{\triangle}$ is well-separated, while the second-largest eigenvalue of $\mathcal{L}_\epsilon^*\mathcal{L}_\epsilon$ is likely to be separated from 1 by order $\epsilon^2$.
If the dynamical system or the domain lacks smoothness, or if the dynamics has nontrivial diffusion from a model, both of which are not uncommon in many real-world applications, then approach of \cite{F13} may be more appropriate.
The nontrivial diffusion from the model will in this case produce a larger spectral gap.
\section{Numerical experiments}
\label{sect:numerics}

In this section, we propose a method for finding coherent sets with low Cheeger ratios $\mathbf{h}^D(\Gamma):=(\ell_{d-1}(\Gamma)+\ell_{d-1}(T\Gamma))/\min\{\ell_d(M_1),\ell_d(M_2)\}$, where $\Gamma$ disconnects $M$ into $M_1,M_2$.
We use the level sets of the first nontrivial eigenfunction of the dynamic Laplacian ${\hat{\triangle}}$, in analogy to the level sets of the first nontrivial eigenfunction of the Laplacian $\triangle$ in the static case described in Section \ref{staticsection}.
Our goal here is to demonstrate the efficacy of this approach, rather than to find the most accurate or efficient numerical implementation, which will be treated in a forthcoming study.

To numerically estimate the Perron-Frobenius operator $\mathcal{P}$ we use Ulam's method \cite{ulam}.
For simplicity, we describe here the case of $T(M)=M$;  the construction for $T(M)\neq M$ is completely analogous and can be found in \cite{FSM10,FPG14}.
We partition $M$ into a grid of $n$ small boxes $\{B_1,\ldots,B_n\}$ and compute a matrix $P$ of conditional transition probabilities between boxes under the action of $T$.
Using a uniform intra-box grid of $Q$ points $z_{i,1},\ldots,z_{i,Q}\in B_i$, one computes $P_{ij}=\#\{z_{i,q}\in B_i: T(z_{i,q})\in B_j\}/\#\{z_{i,q}\in B_i\}$.
The matrix $P$ is a row-stochastic matrix, where the $(i,j)^{\rm th}$ entry estimates the conditional probability of a randomly chosen point in $B_i$ entering $B_j$ under the application of $T$.
The connection with $\mathcal{P}$ is as follows.
Denote by $\pi_n:L^1(M)\to\sp\{\mathbf{1}_{B_1},\ldots,\mathbf{1}_{B_n}\}$ the projection onto characteristic functions on grid sets.
One has $[\pi_n\mathcal{P}f]=\tilde{P}[\pi_nf]$, where $\tilde{P}$ is the transpose of $P$ and $[f]$ denotes the vector formed from the $n$ values taken by an $f\in \sp\{\mathbf{1}_{B_1},\ldots,\mathbf{1}_{B_n}\}$.

The Laplace operator $\triangle$ in our two-dimensional examples is approximated using finite-difference on a five-point stencil, calculated at the centre points of the grid boxes $\{B_1,\ldots,B_n\}$.
We treat the cases where $M$ has boundary rather crudely, simply applying zero Neumann boundary conditions via a symmetric reflection in the finite-difference scheme, without directly enforcing (\ref{strongbc}).
For example, given an $N\times N'$ grid covering a rectangle, denote $f_{i,j}$ to be the value of $f$ at grid position $(x_i,y_j)$.
At the right-hand boundary $f_{N,j}$, we replace the fictional extension $f_{N+1,j}$ in the usual five-point stencil $f_{N+1,j}+f_{N-1,j}+f_{N,j+1}+f_{N,j-1}-4f_{N,j}$ with a symmetric extension $f_{N+1,j}=f_{N-1,j}$ to obtain $2f_{N-1,j}+f_{N,j+1}+f_{N,j-1}-4f_{N,j}$.
The resulting matrix is denoted $\triangle_n$.

We note that the matrices $\tilde{P}$ and $\triangle_n$ are sparse and consequently $\hat{\triangle}_n=\triangle_n+\tilde{P}^\top\triangle_n\tilde{P}$ is also sparse.
The boxes $\{B_1,\ldots,B_n\}$ and matrix $P$ were constructed in Matlab using the GAIO software \cite{DFJ01}.
The level sets of the eigenfunctions of the approximation of $\hat{\triangle}_n$ are extracted automatically using Matlab's \verb"contour" function, with the default settings.

The algorithm we use in the following two-dimensional case studies is described below.
\begin{algorithm}\
\label{mainalg}
\begin{enumerate}
\item Form the matrix $\tilde{P}$ and the discrete Laplacian $\triangle_n$ as described above, and combine to create $\hat{\triangle}_n$.
\item Calculate eigenvalues $\lambda_1>\lambda_2>\cdots$ and eigenvectors $u_1, u_2,\ldots$ of $\hat{\triangle}_n$.
\item Iteratively scan over values of $u_2$ from $\min_i u_{2,i}$ to $\max_i u_{2,i}$. For each value, extract a level curve $\Gamma$ in $M$ using Matlab's \verb"contour" function (this function returns a collection of points representing corners of a polygonal curve).   To compute $T\Gamma$,
\begin{enumerate}
\item Either: map the points representing $\Gamma$ directly with $T$,
\item Or: compute $\tilde{P}u_2$ and extract $T\Gamma$ using Matlab's \verb"contour" function with the same level set value as for $\Gamma$.
\end{enumerate}
\item Optimise $\mathbf{h}^D(\Gamma)$ by running over
all curves $\Gamma$ formed from level sets of $u_2$ in Step 3.
The length of $\Gamma$ and $T\Gamma$ are computed as the lengths of the polygonal curves comprising them. Report the $\Gamma$ and $T\Gamma$ that yield the lowest value of $\mathbf{h}^D(\Gamma)$.
\end{enumerate}
\end{algorithm}

\subsection{Linear shear on a cylinder}

Our first example is a linear shear on a cylinder $M=[0,4)/\sim\times [0,1]$, where the $x$-coordinate is periodic.
The map $T:M\circlearrowleft$ is the horizontal shear $T(x,y)=(x+y,y)$.
We begin by exploring some naive guesses for an optimal $\Gamma$.
Choosing $\Gamma$ to be $\{(x,1/2):0\le x<4\}$ separates the cylinder into upper and lower halves, and such a $\Gamma$ is preserved by $T$;  the length of $\Gamma$ and $T(\Gamma)$ are both relatively long at 4 units each, and $\mathbf{h}^D(\Gamma)=(4+4)/(2\times 2)=2$.
On the other hand, choosing $\Gamma=\{(x,y):0\le y\le 1\}\cup \{(x+2,y):0\le y\le 1\}$ separates the cylinder into two rectangles.
In this case, the length of $\Gamma$ is 2, while the length of $T(\Gamma)$ is $2\sqrt{2}$;  $\mathbf{h}^D(\Gamma)=(2+2\sqrt(2))/(2\times 2)=(1+\sqrt{2})/2$, an improvement over our previous guess.

The numerical computations are carried out using a $256\times 64$ grid of $2^{14}$ square boxes and within each box, $Q=1600$ test points are used to estimate the entries of $P$.
The eigenvalues of $\hat{\triangle}_n$ are $-0.0271, -3.0865, -3.1368, -10.2103, -12.3406, -12.3769,\ldots$.
The first eigenvalue is not exactly zero because the constant vector is not mapped exactly to a constant vector by $\tilde{P}$ due to finite point sampling in its construction.
We use the eigenvector corresponding to $\lambda_2=-3.0865$ to estimate coherent sets.
The results are shown in Figures \ref{shearvecs} and \ref{shearcs}.
\begin{figure}[h!]
  \centering
  \hspace*{-1cm}\includegraphics[width=20cm]{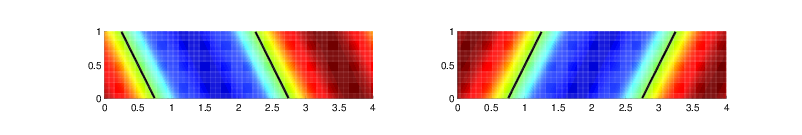}\\
  \caption{Shear map:  The second eigenvector and its image under $\mathcal{P}$ are shown, in addition to the optimised level set at $-1.0457\times 10^{-4}$.}\label{shearvecs}
\end{figure}
\begin{figure}[h!]
  \centering
  \hspace*{-1cm}\includegraphics[width=20cm]{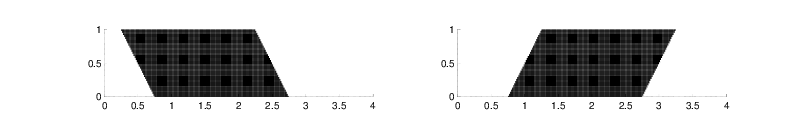}\\
  \caption{Shear map:  The extracted coherent sets at the optimised level set at $- 1.0457\times 10^{-4}$.}\label{shearcs}
\end{figure}

In this simple example, one can calculate exactly that $u_2(x,y)=\sin((x+y/2)\pi/2)$ is an eigenfunction of $\hat{\triangle}$, with eigenvalue $5\pi^2/16$ (multiplicity 2).
One may construct a one-parameter family of optimal coherent sets by sliding the sets in Figure \ref{shearcs} (left) sideways, with corresponding movement of the sets in Figure \ref{shearcs} (right).
The boundaries of the members of this family are of the form $\Gamma=\{(x-y/2,y):0\le y\le 1\}\cup \{(x-y/2+2,y):0\le y\le 1\}$ (parameterised by $x\in [0,4)$).
As the lengths of both $\Gamma$ and $T(\Gamma)$ are both $\sqrt{5}$, we can compute exactly the Cheeger constant using $\ell_d(M_1)=\ell_d(M_2)=2$ to obtain $\mathbf{h}^D(\Gamma)=(\sqrt{5}+\sqrt{5})/(2\times 2)=\sqrt{5}/2$.
Thus, the second eigenfunction of $\hat{\triangle}$ is ``balancing'' the boundary lengths between the initial and final times in order to optimise the sum of these lengths.
Note that $\sqrt{5}/2$ further improves over the Cheeger value $(1+\sqrt{2})/2$ of our second naive solution of above.

We note that the boundary condition (\ref{strongbc}) is automatically satisfied by $u_2(x,y)=\sin((x+y/2)\pi/2)$.
For example, the outward normal vector on the lower boundary of $M$ and $T(M)$ is $\mathbf{n}(x)\equiv [0, -1]^\top$, and $DT^{-1}(x,y)\equiv \begin{pmatrix}1 & -1 \\ 0 & 1 \end{pmatrix}$, so the condition (\ref{strongbc}) is that $\nabla u_2(x,0)\cdot ([0, -1]+[1, -1])^\top=0$, which is clearly satisfied.
The numerically computed eigenfunction in Figure \ref{shearvecs} also appears to satisfy this condition, even with the relatively crude numerical scheme we have employed.

In comparison with the numerics, Algorithm \ref{mainalg} produces $\ell_d(M_1)=\ell_d(M_2)=2$ (to 4 significant figures), while the value for $\ell_{d-1}(\Gamma)+\ell_{d-1}(T\Gamma)$ is around $1\%$ too low because the Matlab's contour function does not extend all the way to the cylinder boundary because of the box discretisation.
The bound for the Cheeger constant from Theorem \ref{cheegerdynthm} is 3.5137, a consistent upper bound for the exact value of $\mathbf{h}^D$.

The remaining eigenfunctions of $\hat{\triangle}$ provide good independent solutions to the dynamic boundary minimising problem.
By Theorem \ref{dynlapthm}, the eigenfunctions of $\hat{\triangle}$ corresponding to distinct eigenvalues are mutually orthogonal.
Thus if we extract coherent sets from different eigenfunctions using the level set approach, we obtain solutions that are ``independent'', in the sense that one is not a small perturbation of another.
In this example, one can exactly compute that $u(x,y)=\cos(\pi y)$ is an eigenfunction with eigenvalue $\pi^2$ (unit multiplicity), and $u(x,y)=\sin((x+y/2)\pi)$ with eigenvalue $5\pi^2/4$ (multiplicity 2).
The eigenvalues $5\pi^2/16, \pi^2,$ and $5\pi^2/4$ are the eigenvalues numbered two to six numerically computed (approximately) above.
The numerically computed eigenfunctions are shown in Figure \ref{shear3efuncs}, and it is clear that zero level sets of these eigenfunctions provide a ranking of good independent solutions of decreasing quality (increasing total boundary length).

\begin{figure}[h!]
  \centering
  \includegraphics[width=12cm]{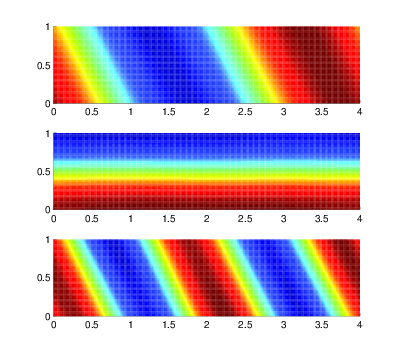}\\
  \caption{Shear map:  The second, fourth, and fifth eigenvectors of $\hat{\triangle}$ (top to bottom).}\label{shear3efuncs}
\end{figure}

\subsection{The standard map on the torus}

Our second example is nonlinear dynamics on a flat boundaryless manifold:  the so-called ``standard map''  $T:\mathbb{T}^2\circlearrowleft$ on the 2-torus is given by $T(x,y)=(x+y,y+8\sin(x+y))\pmod{2\pi}$.
We begin by testing a naive guess for the optimal $\Gamma$, namely one of the continuum of solutions to the static isoperimetric problem illustrated in Figure \ref{torusfig}): $\Gamma=\{(\{\pi/2\}\times [0,2\pi))\cup(\{3\pi/2\}\times [0,2\pi))$.
Figure \ref{fig:standardbenchmark} illustrates the action of $T$ on the partition defined by $\Gamma$;  while the length of $\Gamma$ is short, the nonlinear action of $T$ rapidly lengthens the boundary, and the length of $T(\Gamma)$ is much greater.
\begin{figure}
  \centering
  \includegraphics[width=15cm]{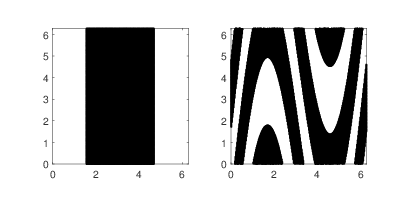}\\
  \caption{Standard map: The black set (left) arises as one of a continuum of solutions to the static isoperimetric problem (see  Figure \ref{torusfig}). Its image (right) has a much longer boundary and consequently a high $\mathbf{h}^D$ value.}\label{fig:standardbenchmark}
\end{figure}

To find the optimal $\Gamma$, numerical computations are carried out using a $128\times 128$ grid of $2^{14}$ boxes and within each box, $Q=1600$ test points are used to estimate the entries of $P$.
The eigenvalues of $\hat{\triangle}_n$ are $-0.1487,-1.6466,-1.6498,-6.0875,-6.0939,\ldots$.
The first eigenvalue is not exactly zero because the constant vector is not mapped exactly to a constant vector by $\tilde{P}$ due to finite point sampling in its construction.
We use the eigenvector corresponding to $\lambda_2=-1.6466$ to estimate coherent sets.
The results are shown in Figures \ref{standardvecs} and \ref{standardcs}.
\begin{figure}[h!]
  \centering
  \includegraphics[width=15cm]{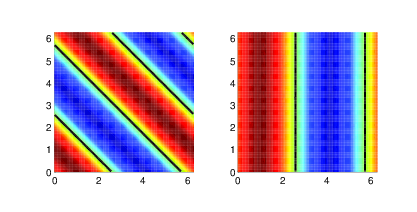}\\
  \caption{Standard map:  The second eigenvector and its image under $\mathcal{P}$ are shown, in addition to the optimised level set at $-2.4741\times 10^{-4}$.}\label{standardvecs}
\end{figure}
\begin{figure}[h!]
  \centering
  \includegraphics[width=15cm]{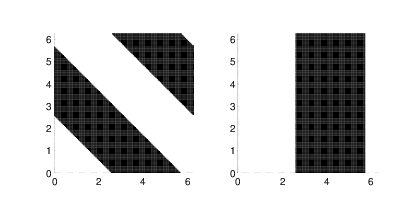}\\
  \caption{Standard map:  The extracted coherent sets at the optimised level set at $-2.4741\times 10^{-4}$.}\label{standardcs}
\end{figure}

It is clear that the fact that the standard map creates affine dynamics in certain directions is being exploited by the operator $\hat{\triangle}$ in order to find boundaries that are initially small and remain small under one iterate of $T$ (in fact, the boundary length is reduced under $T$).
One may construct a one-parameter family of optimal coherent sets by sliding the sets in Figure \ref{standardcs}(b) sideways, with corresponding movement of the sets in Figure \ref{standardcs}(a).
The second eigenvalue of $\hat{\triangle}$ is therefore probably of multiplicity 2, and this is borne out by the closeness of the computed values for $\lambda_2$ and $\lambda_3$.

In this case we can compute exactly the Cheeger constant because $\ell_d(M_1)=\ell_d(M_2)=(2\pi)^2/2$ and $\ell_{d-1}(\Gamma)=4\sqrt{2}\pi$ and $\ell_{d-1}(T\Gamma)=4\pi$.
Thus $\mathbf{h}^D=(1+\sqrt{2})/\pi\approx 0.7685$.
In comparison with the numerics, one obtains $\ell_d(M_1)=\ell_d(M_2)=(2\pi)^2/2$ (to 4 significant figures), while the value for $\ell_{d-1}(\Gamma)+\ell_{d-1}(T\Gamma)$ is around $1\%$ too low because the Matlab's contour function does not extend all the way to the torus boundary.
Bounds for the Cheeger constant from (\ref{soboleveqn}) and Theorem \ref{cheegerdynthm} are 1.2278 and 2.5664, respectively, both consistent upper bounds for the exact value of $\mathbf{h}^D$.

\subsection{Transitory flow on the square}

Our third example is a nonlinear time-dependent flow on the unit square introduced in \cite{meissmosovsky}, defined by $\dot{x}=-\partial\Psi/\partial y, \dot{y}=-\partial\Psi/\partial x,$ where $\Psi$ is the time-dependent stream function $\Psi(x,y,t)=(1-s(t))\sin(2\pi x)\sin(\pi y)+s(t)\sin(\pi x)\sin(2\pi y)$ and  $s(t)=t^2(3-2t), 0\le t\le 1$.
The flow is computed from $t=0$ to $t=1$.
At time $t=0$, the instantaneous vector field comprises two separate rotating ``gyres'' on the left and right halves of the square.
As $t$ increases from 0 to 1, the instantaneous vector field rotates 90 degrees to finally arrive at two rotating gyres in the upper and lower halves of the square.

The numerical computations are carried out using a $128\times 128$ grid of $2^{14}$ boxes and within each box, $Q=1600$ test points are used to estimate the entries of $P$.
The eigenvalues of $\hat{\triangle}_n$ are $-39.9269$, $-87.1430$, $-155.7652$, $-352.8106$, $-430.3017$, $-465.4415,\ldots$
The first eigenvalue is again not exactly zero, because the constant vector is not mapped exactly to a constant vector by $\tilde{P}$ due to finite point sampling in its construction.
We use the eigenvector corresponding to $\lambda_2=-87.1430$ to estimate coherent sets.
The results are shown in Figures \ref{meissvecs} and \ref{meisscs}.

\begin{figure}[h!]
  \centering
  \includegraphics[width=15cm]{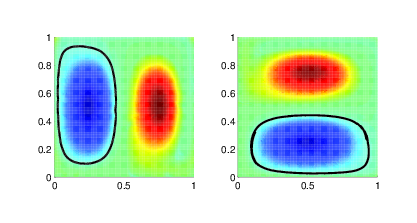}\\
  \caption{Transitory flow:  The second eigenvector and its image under $\mathcal{P}$ are shown, in addition to the optimised level set at $-6.4417\times 10^{-4}$.}\label{meissvecs}
\end{figure}
\begin{figure}[h!]
  \centering
  \includegraphics[width=15cm]{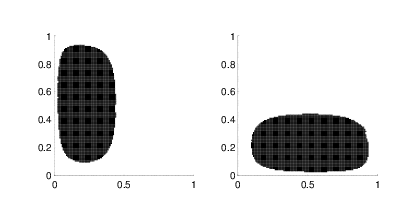}\\
  \caption{Transitory flow:  The extracted coherent sets at the optimised level set at $-6.4417\times 10^{-4}$.}\label{meisscs}
\end{figure}

From the numerics, one obtains $\ell_d(M_1)=0.3091$, $\ell_{d-1}(\Gamma)=2.1606$, $\ell_{d-1}(T\Gamma)=2.9557$ and the value for $\mathbf{h}^D(\Gamma)\approx(2.1606+2.9557)/(2\times 0.3091)=8.2749$.
Bounds for the Cheeger constant from (\ref{soboleveqn}) and Theorem \ref{cheegerdynthm} are 10.0533 and 18.6701, respectively, both consistent upper bounds.

We compare these results with a ``naive'' solution, where one selects $\Gamma'$ to be the vertical separatrix that separates the two rotating elements in the instantaneous vector field at $t=0$;  see Figure \ref{meissvi} (left).
This choice of $\Gamma'$ is one of two solutions to the static isoperimetric problem on the unit square, and corresponds to a static Cheeger value of $\mathbf{h}(\Gamma')=1/(1/2)=2.$
The image of $\Gamma'$ under $T$ is shown in Figure \ref{meissvi} (right).
\begin{figure}[h!]
  \centering
  \includegraphics[width=15cm]{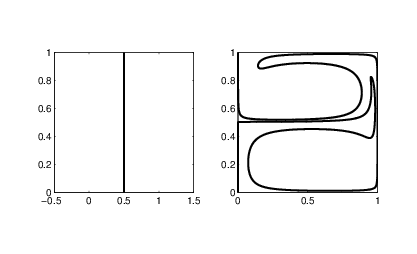}\\
  \caption{Transitory flow:  A vertical separatrix and its image from $t=0$ to $t=1$.}\label{meissvi}
\end{figure}
While the length of $\Gamma'$ is only 1, the length of $T\Gamma'$ is much greater (approximately 8.3057), leading to a Cheeger value of  $\mathbf{h}^D(\Gamma')\approx(1+8.3057)/(2\times 1/2)=9.3057$,  larger than the value of $\mathbf{h}^D(\Gamma)=8.2749$ corresponding to the solution shown in Figures \ref{meissvecs} and \ref{meisscs}.
We see that the curve $\Gamma$ in Figure \ref{meissvecs} trades off length at $t=0$ in order to have a relatively short length also at time $t=1$, in contrast to $\Gamma'$.

Finally, Figure \ref{meisszooms} shows fine detail of the curves $T(\Gamma)$;  the pixellation visible is the underlying grid, which controls the resolution of the boundary curves.
There is some shearing at the two locations shown.
This is responsible for most of the increase in $\ell_{d-1}(T(\Gamma))$ from $\ell_{d-1}(\Gamma)$.
\begin{figure}[h!]
  \centering
  \includegraphics[width=15cm]{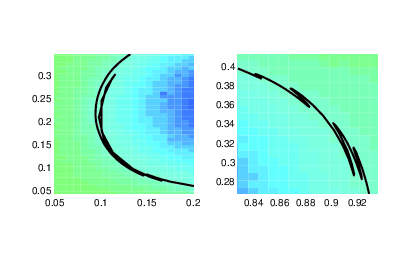}\\
  \caption{Transitory flow:  Zooms of the boundary at time $t=1$.}\label{meisszooms}
\end{figure}
While the shearing is not tiny, particularly in the left-hand figure, given the limited resolution and the fact that most of the boundary is shear-free, our selected coherent sets do perform well in terms of reducing boundary length for both the initial set and its image.
Moreover, if one considers applying diffusion at the scale of the box diameters, the ``effective boundary'' at this scale (responsible for possible diffusive ejection as discussed in \S5) is increased only a little by the tight shearing.


\section{Conclusion}
We have extended classical results from isoperimetric theory, concerned with identifying subsets of manifolds with least boundary size to volume ratios, to the situation where the manifolds are subjected to general nonlinear dynamics.
We proved a dynamic version of (i) the Federer-Fleming Theorem, which tightly links geometric and functional characterisations of the fundamental isoperimetric problem, and (ii) the Cheeger inequality, which bounds the least boundary size to volume ratio by the first nontrivial eigenvalue of the Laplace operator on the manifold.
We developed a new dynamic Laplace operator and used this operator to numerically identify subsets of manifolds that have small boundary size to volume ratios before, after, and during, the application of nonlinear dynamics.
In nonlinear fluid flow, such sets characterise finite-time coherent sets, as their boundaries do not elongate and filament, and there is little exchange between the interior and exterior of these sets in the presence of small diffusion.
We proved that the dynamic Laplace operator can also be obtained as a zero-diffusion limit of the existing probabilistic approach to identifying finite-time coherent sets \cite{F13}, thus creating a strong formal link between probabilistic descriptions and geometric descriptions of Lagrangian coherent structures.
Numerical experiments were carried out using a simple combination of Ulam's method and a finite-difference scheme.

Obvious extensions of the methodology include handling non-volume-preserving dynamics, nonuniform initial mass distributions, and manifolds of nonvanishing curvature, and work is in progress in these directions.
Accurate and efficient numerical methods are also being pursued.
An advantage of the present formulation over \cite{F13} in the pure advection setting is that there is more freedom in selecting an approximating function basis as the basis no longer needs to generate numerical diffusion, and various out-of-the-box numerical methods can be employed.
Recent work \cite{FJ15} uses radial basis functions to estimate both $\mathcal{P}$ and $\triangle$ and has resulted in a more accurate approximation of the eigenspectrum and a significant reduction of the number of required Lagrangian trajectories, compared to the numerical techniques in the present paper.
Radial basis functions are flexible enough to be able to handle irregularly-shaped domains as sometimes arise in applications.

\section{Acknowledgements}
The author acknowledges feedback from Eric Kwok and Daniel Karrasch, which improved the manuscript, assistance from Oliver Junge regarding GAIO, and a discussion with Renato Feres. This research is supported by an Australian Research Council Future Fellowship and Discovery Project DP150100017.




\appendix
\section{Proof of Theorem \ref{fflemma}}

\begin{lemma}
\label{sl}
Let $A\in GL(d)$, and $v_1,\ldots, v_d$ be an orthonormal basis for $\mathbb{R}^d$. Let $U_1=\sp\{v_1,\ldots, v_k\}, U_2=\sp\{v_{k+1},\ldots,v_{d}\}$.
Then $$\|A(v_1\wedge \cdots \wedge v_k)\|=|\det(A)|\cdot\|(A^{-1})^\top(v_{k+1}\wedge\cdots \wedge v_d)\|,$$
where $\|\cdot\|$ is the volume induced by the Gram determinant.
\end{lemma}
\begin{proof}
The parallelopiped $A(v_1\wedge\cdots\wedge v_d)$ has volume $|\det(A)|$ by orthonormality of $v_1,\ldots,v_d$.
We note that the space spanned by $(A^{-1})^\top v_{k+1},\ldots, (A^{-1})^\top v_d$ is orthogonal to the space spanned by $Av_1,\ldots A v_k$;  indeed any element of one collection is orthogonal to any element of the other.
The volume of the parallelopiped can therefore be written as $\det(A)=\|Av_1\wedge\cdots\wedge A v_k\|\cdot\|\Pr_{(A^{-1})^\top(U_2)}(Av_{k+1})\wedge\cdots\wedge \Pr_{(A^{-1})^\top(U_2)}(A v_d)\|$, where $\Pr_{(A^{-1})^\top(U_2)}$ denotes orthogonal projection along $A(U_1)$ onto $(A^{-1})^\top(U_2)$.
Let $V$ be the $d\times (d-k)$ matrix with columns $v_{k+1},\ldots,v_d$, and let $W=(A^{-1})^\top V$.
The projection matrix associated with $\Pr_{(A^{-1})^\top(U_2)}$ is $C=W(W^\top W)^{-1}W^\top$.
We compute $\|\Pr_{(A^{-1})^\top(U_2)}(Av_{k+1})\wedge\cdots\wedge \Pr_{(A^{-1})^\top(U_2)}(A v_d)\|$ as $\det((CAV)^\top CAV)^{1/2}$.
\begin{eqnarray*}
\lefteqn{\det((CAV)^\top CAV)}\\
&=&\det\left(\left[V^\top A^\top (A^{-1})^\top V(V^\top A^{-1}(A^{-1})^\top V)^{-1}V^\top A^{-1}\right] \left[(A^{-1})^\top V(V^\top A^{-1}(A^{-1})^\top V)^{-1}V^\top A^{-1}AV\right]\right)\\
&=&\det\left(\left(V^\top A^{-1}(A^{-1})^\top V\right)^{-1}\left(V^\top A^{-1}(A^{-1})^\top V\right)\left(V^\top A^{-1}(A^{-1})^\top V\right)^{-1}\right)\quad\mbox{by orthogonality of $V$}\\
&=&\det(V^\top A^{-1}(A^{-1})^\top V)^{-1})\\
&=&1/\|(A^{-1})^\top (v_{k+1}\wedge\cdots\wedge v_d)\|^2
\end{eqnarray*}
Thus, $\|\Pr_{(A^{-1})^\top(U_2)}(Av_{k+1})\wedge\cdots\wedge \Pr_{(A^{-1})^\top(U_2)}(A v_d)\|=1/\|(A^{-1})^\top(v_{k+1}\wedge\cdots \wedge v_d)\|$, and the result follows.
\end{proof}

\begin{proof}[Proof of Theorem \ref{fflemma}]
The main thing to prove is the equality.
We modify the arguments of Remark VI.2.3 and the proof of Theorem II.2.1\cite{chavelisoperimetric}.

(a)
We start by showing $\mathbf{s}^D\le \mathbf{h}^D$.
We do this by creating a specific sequence of functions $f_\epsilon$, which when substituted into (\ref{soboleveqn}), in the limit achieve $\mathbf{h}^D$;  therefore $\mathbf{s}^D$ can potentially be lower still.
Suppose we have a specific disconnection $\Gamma$, and define $\Gamma_\epsilon=\{x\in M: d(x,\Gamma)<\epsilon\}$, where $d(x,\Gamma)=\inf_{y\in \Gamma}\|x-y\|$, and because of the vanishing curvature we write the Riemannian distance between two points $x,y\in M$ as $\|x-y\|$.
Define
$$f_\epsilon=\left\{
               \begin{array}{ll}
                 1, & \hbox{$x\in M_1\setminus\Gamma_\epsilon$;} \\
                -1, & \hbox{$x\in M_2\setminus\Gamma_\epsilon;$}\\
 (1/\epsilon)d(x,\Gamma),&\hbox{$x\in M_1\cap\Gamma_\epsilon$;}\\                -(1/\epsilon)d(x,\Gamma),&\hbox{$x\in M_2\cap\Gamma_\epsilon.$}
               \end{array}
             \right.
$$
The function $f_\epsilon$ is Lipschitz and by mollification on $\mathring{M}$ we can produce a sequence of  $C^\infty$ functions $\phi_{j,\epsilon}$ such that $\| f_\epsilon- \phi_{j,\epsilon}\|_1\to 0$ and $\|\nabla f_\epsilon-\nabla \phi_{j,\epsilon}\|_1\to 0$ as $j\to \infty$ (see e.g. Theorem I.3.3 \cite{chavelisoperimetric}).
Now,
\begin{eqnarray*}
\mathbf{s}^D&=&\inf_{f\in C^\infty} \frac{\|\nabla f\|_1+\|\nabla\P f\|_1}{2\inf_\alpha \|f-\alpha\|_{1}}\\
&\le&\frac{\|\nabla \phi_{j,\epsilon}\|_1+\|\nabla\P \phi_{j,\epsilon}\|_1}{2\inf_\alpha \|\phi_{j,\epsilon}-\alpha\|_{1}}\qquad\mbox{for each $j$}\\
&\le&\frac{\|\nabla f_\epsilon\|_1+\|\nabla \phi_{j,\epsilon}-\nabla f_\epsilon\|_1+\|\nabla\P f_\epsilon\|_1+\|\nabla\P \phi_{j,\epsilon}-\nabla\P f_\epsilon\|_1}{2\inf_\alpha \|f_\epsilon-\alpha\|_{1}-2\|f_\epsilon-\phi_{j,\epsilon}\|_{1}}\qquad\mbox{for each $j$}\\
\end{eqnarray*}
Thus, letting $j\to\infty$ we have for each $\epsilon>0$,
\begin{equation}
\label{sDcompare1}
\mathbf{s}^D\le \frac{\|\nabla f_\epsilon\|_1+\|\nabla\P f_\epsilon\|_1}{2\inf_\alpha \|f_\epsilon-\alpha\|_{1}}.
\end{equation}
We begin to interpret these terms in terms of $d$- and $d-1$-dimensional volume.
Note that $|\nabla f_\epsilon|$ is $1/\epsilon$ on $\Gamma_\epsilon$ and zero elsewhere.
Thus $\lim_{\epsilon\to 0}\int_M |\nabla f_\epsilon|\ d\ell=\lim_{\epsilon\to 0}\ell(\Gamma_\epsilon)/\epsilon=2\ell_{d-1}(\Gamma)$.

Now we concentrate on the term $\|\nabla \P f_\epsilon\|_1$.
Let $x\in\Gamma_\epsilon\cap M_2$, and $z\in\Gamma$ be the closest point to $x$ (if there are several, choose one).
Note $\nabla f_\epsilon(x)=\hat{n}(x)/\epsilon$ where $\hat{n}(x)=(z-x)/|z-x|$, which is normal to $\Gamma$ at $z$.
Since $T$ is volume-preserving we note that $\P f_\epsilon$ is 1 on $T( M_1\setminus\Gamma_\epsilon)$ and $-1$ on $T( M_2\setminus\Gamma_\epsilon)$.
Thus, $|\nabla (\P f_\epsilon)|=0$ on these regions.
The value of $\P f_\epsilon$ on $T\Gamma_\epsilon$ must be computed.
Let us first consider $T( M_2\cap\Gamma_\epsilon)$.
\begin{eqnarray}
\nonumber\int_{T( M_2\cap\Gamma_\epsilon)}|\nabla (f_\epsilon\circ T^{-1})(x)|\ d\ell&=&\int_{T( M_2\cap\Gamma_\epsilon)}|\nabla f_\epsilon(T^{-1}x)^\top\cdot DT^{-1}(x)|\ d\ell\\
\nonumber&=&(1/\epsilon)\int_{T( M_2\cap\Gamma_\epsilon)}|\hat{n}(T^{-1}x)^\top\cdot DT^{-1}(x)|\ d\ell\\
\nonumber&=&(1/\epsilon)\int_{ M_2\cap\Gamma_\epsilon}|\hat{n}(x)^\top\cdot DT^{-1}(Tx)|\ d\ell\\
\label{final}&=&(1/\epsilon)\int_{ M_2\cap\Gamma_\epsilon}|(DT(x)^{-1})^\top\hat{n}(x)|\ d\ell
\end{eqnarray}

Let $t_1(x),\ldots,t_{d-1}(x)$ be an orthonormal set of vectors spanning the orthogonal complement of $\hat{n}(x)$ in $\mathbb{R}^d$ (these vectors span the $d-1$-dimensional tangent space of $\Gamma$ at $z$).
By Lemma \ref{sl}, one has $|(DT(x)^{-1})^\top\hat{n}(x)|=|DT(x)(t_1(x)\wedge\cdots \wedge t_{d-1}(x))|$, where $|\cdot|$ denotes the volume (one-dimensional and $d-1$-dimensional, respectively) induced by the Gram determinant.
Thus,
$$(\ref{final})=(1/\epsilon)\int_{ M_2\cap\Gamma_\epsilon}|DT(x)(t_1(x)\wedge\cdots \wedge t_{d-1}(x))|\ d\ell.$$
The integrand measures the local increase in the $d-1$-dimensional volume of linear spaces close to the tangent spaces of $\Gamma$, under the action of $T$ in an $\epsilon$-neighbourhood of $\Gamma$, and the above integral converges to $\ell_{d-1}(T\Gamma)$ as $\epsilon\to 0$.
Similarly,
$$\lim_{\epsilon\to 0}\int_{T( M_1\cap\Gamma_\epsilon)}|\nabla (f_\epsilon\circ T^{-1})(x)|\ d\ell=\ell_{d-1}(T\Gamma).$$
Thus,
\begin{equation}
\label{sDcompare2}
\lim_{\epsilon\to 0} (\|\nabla f_\epsilon\|_1+\|\nabla \P f_\epsilon\|_1)/2=\ell_{d-1}(\Gamma)+\ell_{d-1}(T\Gamma).
\end{equation}

Now we turn to the denominator $\int_M |f_\epsilon-\alpha|\ d\ell$.
Without loss, suppose that $\ell( M_1)\le \ell( M_2)$.
\begin{eqnarray*}
\int_M |f_\epsilon-\alpha|\ d\ell&\ge& |1-\alpha|(\ell( M_1)-\ell(\Gamma_\epsilon))+|1+\alpha|(\ell( M_2)-\ell(\Gamma_\epsilon))\\
&\ge&(|1-\alpha|+|1+\alpha|)(\ell( M_1)-\ell(\Gamma_\epsilon))\\
&\ge&2(\ell( M_1)-\ell(\Gamma_\epsilon)),
\end{eqnarray*}
implying $\inf_\alpha \int_M |f_\epsilon-\alpha|\ d\ell\ge 2(\ell( M_1)-\ell(\Gamma_\epsilon))$ for each $\epsilon>0$.
Taking the limit as $\epsilon\to 0$, we combine this with (\ref{sDcompare1}) and (\ref{sDcompare2})
to conclude $\mathbf{s}^D\le \mathbf{h}^D$.

(b)
Now let $f\in C^\infty(M)$ and choose a constant $\beta$ so that, $ M_1=\{f>\beta\}$, $ M_2=\{f<\beta\}$ have equal volume.
Such a choice of $\beta$ satisfies $\|f-\beta\|_1=\inf_\alpha \|f-\alpha\|_1$ (see Remark  VI.2.2 p163 \cite{chavelisoperimetric}).
For $t>0$ define $D_t=\{x\in  M_1: f(x)>\beta+t\}$, and $\tilde{D}_t=\{x\in T(M_1): \P f(x)>\beta+t\}=
\{x\in T(M_1): f\circ T^{-1}(x)>\beta+t\}$, thus $\tilde{D}_t=TD_t$.
In what follows, we concentrate on $\tilde{D}_t$ and $\P f$, modifying the argument for $D_t$ and $f$ in \cite{chavelisoperimetric} p46.
Firstly, using the co-area formula, Corollary I.3.1 \cite{chavelisoperimetric} with $f\equiv 1$, $\Phi=f-\beta$ (and then $\Phi=\mathcal{P}f-\beta$), one has
\begin{equation}
\label{coarea1}
\int_{ M_1}|\nabla (f -\beta)|\ d\ell+\int_{T(M_1)}|\nabla (\P f -\beta)|\ d\ell=\int_0^\infty (\ell_{d-1}(\partial{D}_t)+\ell_{d-1}(\partial\tilde{D}_t))\ dt.
\end{equation}
Continuing,
\begin{eqnarray}
\label{altsobvol}
(\ref{coarea1})&\ge& 2\mathbf{h}^D \int_0^\infty \ell(D_t)\ dt\quad\mbox{since $\ell(D_t)=\ell(\tilde{D}_t)\le \ell(M)/2$}\\
\label{hfinal}&=& 2\mathbf{h}^D \int_{ M_1} |f-\beta|\ d\ell,
\end{eqnarray}
by a standard argument, see e.g.\ p.164 \cite{chavelisoperimetric}.
Similarly, $\int_{ M_2}|\nabla (f -\beta)|\ d\ell+\int_{T(M_2)}|\nabla (\P f -\beta)|\ d\ell\ge 2\mathbf{h}^D \int_{ M_2} |f-\beta|\ d\ell$.
Thus,
\begin{eqnarray}
\label{altsob1}\int_{ M}|\nabla f|\ d\ell+\int_{T(M)}|\nabla \P f |\ d\ell&=&\int_{ M}|\nabla (f -\beta)|\ d\ell+\int_{T(M)}|\nabla (\P f -\beta)|\ d\ell\\
\nonumber&\ge& 2\mathbf{h}^D \int_{ M} |f-\beta|\ d\ell\\
\nonumber&\ge& 2\mathbf{h}^D \inf_\alpha \int_{ M} |f-\alpha|\ d\ell
\end{eqnarray}
and $\mathbf{s}^D\ge \mathbf{h}^D$.


(c) 
In order to get the inequality $\mathbf{h}^D/2\le \hat{\mathbf{s}}^D$, we set $\beta=\bar{f}$ in the argument of part (b) above.
Note that now possibly only one of $M_1$ or $M_2$ has volume less than or equal to $\ell(M)/2$, and WLOG suppose it is $M_1$.
Then (\ref{hfinal}) holds.
We note that $\int_{T(M_2)}|\mathcal{P}f-\beta|\ d\ell=\int_{M_2} |f-\beta|\ d\ell\ge \int_{M_1} |f-\beta|\ d\ell=\int_{T(M_1)}|\mathcal{P}f-\beta|\ d\ell$.
Thus,
\begin{eqnarray}
\frac{\int_M |\nabla(f-\beta)|\ d\ell +\int_{T(M)}|\nabla(\mathcal{P}f-\beta)|\ d\ell}{2\int_M|f-\beta|\ d\ell}&\ge&\frac{\int_{M_1} |\nabla(f-\beta)|\ d\ell +\int_{T(M_1)}|\nabla(\mathcal{P}f-\beta)|\ d\ell}{4\int_{M_1}|f-\beta|\ d\ell}\\
&\ge&\mathbf{h}^D/2,
\end{eqnarray}
by (\ref{hfinal}).
%
%
%
%
Thus, $\mathbf{h}^D\le 2\hat{\mathbf{s}}^D$.
\end{proof}

\begin{proof}[Proof of Corollary \ref{ffcor}]
All of the calculations concerning the map $T$ and the operator $\mathcal{P}$ in the proof of Theorem \ref{ffthm} hold for each of the maps $T^{(i)}, i=1,\ldots,n$ or $T^{(t)}, t\in[0,\tau]$.
These calculations are always put together linearly, and the proof proceeds exactly as in the proof of  Theorem \ref{ffthm}.
\end{proof}

\section{Proof of Theorem \ref{cheegerdynthm}}

\begin{proof}[Proof of Theorem \ref{cheegerdynthm}]

The proof is a modification of the presentation of \cite{ledoux};  see also \cite{chavelisoperimetric} Theorem 3, Section IV.3.
Let $g:M\to \mathbb{R}$ be positive and smooth;  then $\P g$ is also positive and smooth.
First, by the co-area formula applied separately to $g$ and $\P g$ (see e.g.\ Cor.\ I.3.1 \cite{chaveleigenvalues}) and then the definition of $\mathbf{h}^D$ we have that
\begin{eqnarray}
\label{coareaeqn}\int_M |\nabla g|\ d\ell+\int_{T(M)}|\nabla(\P g)|\ d\ell&=&\int^\infty_0 \ell_{d-1}(\{g=t\})+\ell_{d-1}(\{\P g=t\})\ dt \\
&=&\int^\infty_0 \ell_{d-1}(\{g=t\})+\ell_{d-1}(T\{g=t\})\ dt \\
\label{heqn}&\ge& 2\mathbf{h}^D\int_0^\infty \min\{\ell(\{g\ge t\}),\ell(\{g<t\})\}\ dt
\end{eqnarray}

Let $f:M\to \mathbb{R}$ be smooth and denote by $m$ the median of $f$;  i.e.\ $\ell(f\ge m)\ge 1/2$ and $\ell(f\le m)\ge 1/2$.
Set $f^+=\max\{f-m,0\}, f^-=\max\{-f+m,0\}$, so that $f-m=f^+-f^-$.
Note that by volume-preservation, $m$ is also the median for $\P f$, and we similarly decompose $\P f-m=(\P f)^+-(\P f)^-$.
Further, note that since $\P$ is positive and a composition operator, we have $((\P f)^+)^2=\P ((f^+)^2)$ and similarly for $f^-$.
We apply (\ref{heqn}) to $g=(f^+)^2$ and $g=(f^-)^2$.
Note that for each $t>0$, $\ell(\{(f^+)^2\ge t\})\le 1/2$ and $\ell(\{(f^-)^2\ge t\})\le 1/2$.
Now,
\begin{eqnarray}
\nonumber\lefteqn{\frac{1}{2}\left(\int_M |\nabla((f-m)^2)|\ d\ell+\int_{T(M)}|\nabla((\P f-m)^2)|\ d\ell\right)}\\
\nonumber&=&\frac{1}{2}\left(\int_M |\nabla((f^+)^2)|+|\nabla((f^-)^2)|\ d\ell+\int_{T(M)}|\nabla((\P f^+)^2)|+|\nabla((\P f^-)^2)|\ d\ell\right)\\
\nonumber&\ge&\mathbf{h}^{D}\int^\infty_0 \ell(\{(f^+)^2\ge t\})\ dt+\mathbf{h}^{D}\int^\infty_0 \ell(\{(f^-)^2\ge t\})\ dt\\
\nonumber&=&\mathbf{h}^{D}\int_M (f^+)^2\ d\ell+\mathbf{h}^{D}\int_M (f^-)^2\ d\ell\\
\label{eqn1}&=&\mathbf{h}^{D}\int_M (f-m)^2\ d\ell
\end{eqnarray}
Further,
\begin{eqnarray}
\nonumber\frac{1}{2}\int_M|\nabla((f-m)^2)|\ d\ell&=&\int_M|(f-m)\cdot\nabla f|\ d\ell\\
\label{cauchyschwartz}&\le& \|f-m\|_2\cdot\|\nabla f\|_2,
\end{eqnarray}
where $\|\cdot\|_2$ denotes the $L^2(\ell)$ norm.
Also analogously to (\ref{cauchyschwartz}) we have
\begin{equation}
\label{cauchyschwartza}
\frac{1}{2}\int_{T(M)}|\nabla((\P f-m)^2)|\ d\ell\le \|\P f-m\|_2\cdot\|\nabla (\P f)\|_2=\|f-m\|_2\cdot\|\nabla (\P f)\|_2.
\end{equation}
Thus, using (\ref{eqn1})--(\ref{cauchyschwartza}) and Cauchy-Schwartz,
\begin{eqnarray}
\nonumber
(\mathbf{h}^{D})^2\|(f-m)\|_2^4&\le&\|f-m\|_2^2\left(\|\nabla f\|_2+\|\nabla(\P f)\|_2\right)^2\\
\label{penultimate}&\le&2\|f-m\|_2^2\left(\|\nabla f\|_2^2+\|\nabla(\P f)\|_2^2\right).
\end{eqnarray}
Thus (\ref{penultimate}) becomes
\begin{eqnarray}
\nonumber(\mathbf{h}^{D})^2&\le &2\frac{\int_M|\nabla f|^2\ d\ell+\int_{T(M)}|\nabla(\P f)|^2\ d\ell}{\int_M (f-m)^2\ d\ell}\\
\label{cheegerexpr}&\le &
2\frac{\int_M|\nabla f|^2\ d\ell+\int_{T(M)}|\nabla(\P f)|^2\ d\ell}{\int_M (f-(\int_M f\ d\ell))^2\ d\ell},
\end{eqnarray}
since $\inf_\alpha\|f-\alpha\|_2$ is realised when $\alpha$ is the mean of $f$.
As $f$ is arbitrary,
 we may minimise the RHS of (\ref{penultimate}) by inserting $f=u_2$, the eigenfunction of $\hat{\triangle}$ corresponding to the lowest nontrivial eigenvalue (with $u_2$ satisfying the boundary condition of Theorem \ref{cheegerdynthm} if $M$ has nonempty boundary).
Upon this insertion, the RHS of (\ref{penultimate}) takes the value $4\lambda_2$ by Part 3 of Theorem \ref{dynlapthm}.
\end{proof}

\begin{proof}[Proof of Corollaries \ref{cheegerncor} and \ref{cheegertcor}]
All of the calculations concerning the map $T$ and the operator $\mathcal{P}$ in the proof of Theorem \ref{cheegerdynthm} hold for each of the maps $T^{(i)}, i=1,\ldots,n$ or $T^{(t)}, t\in[0,\tau]$.
These calculations are almost always put together linearly;  the only exception is equation (\ref{penultimate}), which we describe here in the continuous time case.
\begin{eqnarray}
\nonumber
(\mathbf{h}^{D})^2\|(f-m)\|_2^4&\le&\frac{1}{\tau^2}\left(\int_0^\tau 2\|f-m\|_2\cdot\|\nabla \P^{(t)}f\|_2\ dt\right)^2\\
\label{penultimatet}&\le&\frac{4\|f-m\|_2^2}{\tau^2}\int_0^\tau \|\nabla \P^{(t)}f\|_2^2\ dt\cdot\int_0^\tau 1\ dt.
\end{eqnarray}
Thus $$(\mathbf{h}^{D})^2\le 4\frac{\frac{1}{\tau}\int_0^\tau \|\nabla \P^{(t)}f\|_2^2\ dt}{\|f-m\|^2_2}.$$
The proof proceeds exactly as in the proof of  Theorem \ref{cheegerdynthm}, using Remark \ref{spectrummultistep}.
\end{proof}

\section{Proof of Theorem \ref{dynlapthm}}

\subsection{Existence of weak solutions and variational characterisation of eigenvalues}
\label{sec:weakexist}
Let $X=W^{1,2}(M)$, the Sobolev space of functions $u:M\to\mathbb{R}$ with square-integrable weak derivative.
The space $X$ is a Hilbert space with the inner product $\langle u,v\rangle_X=\int_M \nabla u\cdot \nabla v+uv\ d\ell$.
We will establish the existence of a set of weak solutions $u\in X$ to
\begin{equation}
\label{weakeqn}
(1/2)\left(\int_M \nabla v\cdot \nabla u\ d\ell+\int_{T(M)} \nabla (\P v)\cdot \nabla (\P u)\ d\ell\right)=-\lambda \int_M v\cdot u\ d\ell\qquad\mbox{for all $v\in W^{1,2}$}.
\end{equation}
We require certain extremisation properties and therefore for $u\in X$, we define the functionals $F(u)=F_1(u)+F_2(u)$, where $F_1(u)=(1/2)\int_M |\nabla u|^2\ d\ell$, $F_2(u)=(1/2)\int_{T(M)}|\nabla(\P u)|^2\ d\ell$ and $G(u)=\int_M u^2\ d\ell-1$.
We look for $u$ which minimizes $F(u)$ subject to $G(u)=0$ ($\|u\|_2=1$).
In the following, we consider only $F_2(u)$ as the corresponding results for $F_1(u)$ follow immediately by setting $T$ to the identity map and $\P$ the identity operator.
\begin{lemma}
\label{functionallemma}\
\begin{enumerate}
\item[(i)] The functional $F_2:X\to\mathbb{R}$ is well-defined,
    \item[(ii)] The derivative $F_2'(u)$ is linear and bounded (hence $F_2'(u)\in X^*$),
        \item[(iii)] $F_2$ is differentiable, and
         \item[(iv)] $u\mapsto F'_2(u)$ is continuous as a map from $X$ to $X^*$.
             \end{enumerate}
\end{lemma}
\begin{proof}\quad
\begin{itemize}
\item[(i)]
$2F_2(u)=\int_{T(M)} \|\nabla(\P u)\|^2\ d\ell=\int_{T(M)} \|\nabla(u\circ T^{-1})\|^2\ d\ell=\int_{T(M)} \|(DT^{-1})^\top\cdot(\nabla u)\circ T^{-1}\|^2\ d\ell$.
    By compactness of $M$ and the fact $T$ is $C^\infty$ diffeomorphism onto $T(M)$, one may find a $C<\infty$ such that the previous expression is bounded above by $C\int\|\nabla u\|^2\ d\ell\le C\|u\|_X^2<\infty$.
Thus, $F_2:X\to \mathbb{R}$ is well-defined.
\item[(ii)]
\begin{eqnarray*}
F'_2(u)v&=&\lim_{h\to 0}\frac{F_2(u+hv)-F_2(u)}{h}\\
&=&\lim_{h\to 0}\frac{\int_{T(M)}|\nabla(\P u)+h\nabla (\P v)|^2\ d\ell-\int_{T(M)}|\nabla (\P u)|^2\ d\ell}{2h}\\
&=&\lim_{h\to 0}\frac{\int_{T(M)}|\nabla(\P u)|^2+2h(\nabla(\P u)\cdot\nabla (\P v))+h^2|\nabla (\P v)|^2 -|\nabla (\P u)|^2\ d\ell}{2h}\\
&=&\lim_{h\to 0}\frac{\int_{T(M)} 2h(\nabla(\P u)\cdot\nabla (\P v))+h^2|\nabla (\P v)|^2\  d\ell}{2h}\\
&=&\int_{T(M)} \nabla(\P u)\cdot\nabla (\P v)\ d\ell.
\end{eqnarray*}
$F'_2(u)$ is clearly linear
in $v$ and bounded because $\int_{T(M)} \nabla(\P u)\cdot\nabla (\P v)\ d\ell\le \|\nabla(\P u)\|_2\cdot\|\nabla (\P v)\|_2\le C^{1/2}\|u\|_X\|v\|_X$, where $C$ is the constant from part (i);  thus $F_2'(u)\in X^*$.
\item[(iii)]
$F_2$ is differentiable since
\begin{eqnarray*}
|F_2(u+v)-F_2(u)-F_2'(u)v|&=&\left|(1/2)\int_{T(M)}|\nabla(\P v)|^2\ d\ell\right|\\
&\le& (C/2)\int_M |\nabla v|^2\ d\ell\\
&\le& (C/2)\|v\|_X^2\to 0\mbox{ as }\|v\|_X\to 0.
\end{eqnarray*}
\item[(iv)]
Finally, let $u,v,w\in X$, then
\begin{eqnarray*}|(F'_2(u)-F'_2(v))w|&=&\left|\int_{T(M)} (\nabla(\P u)-\nabla(\P v))\cdot\nabla(\P w)\ d\ell\right|\\
&\le& \|\nabla(\P(u-v))\|_2\cdot\|\nabla(\P w)\|_2\\
&\le& C\|\nabla(u-v)\|_2\cdot\|\nabla w\|_2\\
&\le& C\|u-v\|_X\cdot\|w\|_X.
\end{eqnarray*}
Thus $\|F'_2(u)-F'_2(v)\|_{X^*}=\sup_{\|w\|_X=1}|(F'_2(u)-F'_2(v))w|\to 0$ as $\|u-v\|_X\to 0$, and $F'_2:X\to X^*$ is continuous.
\end{itemize}
\end{proof}

\begin{lemma}
\label{lemmaattain}
$F$ attains its minimum on the constraint set $\mathcal{C}=\{u\in X: G(u)=0\}$.
\end{lemma}
\begin{proof}
Let $I=\inf_{u\in X}\{F(u):G(u)=0\}$. Select a sequence $u_j\in\mathcal{C}$ such that $F(u_j)\to I$ and $F(u_j)\le I+1$ for all $j\ge 0$.
By the Poincar\'{e} inequality (e.g.\ p163 \cite{mcowen}), the norm $\|\cdot\|_X$ is equivalent to $\|\nabla(\cdot)\|_2$, so using the form of $F$, the $u_j$ are uniformly bounded in norm in $X$.
By standard arguments using Rellich compactness (e.g.\ Thm.\ 8.4.2 \cite{jost}), one can find a subsequence $u_{j_k}\in X$ and a function $\bar{u}\in X$ such that $u_{j_k}\to \bar{u}$ in $L^2(M)$.
For the weak derivatives, we develop a Cauchy sequence.
We first demonstrate that there is a $c>0$ such that $\|\nabla ( u_{j_k}-u_{j_l})\|\le (1/c)\|\nabla(\mathcal{P}u_{j_k}-\mathcal{P}u_{j_l})\|$.
\begin{eqnarray*}
\|\nabla(\mathcal{P}u_{j_k}-\mathcal{P}u_{j_l})\|^2&=&\int_{T(M)}|\nabla(u_{j_k}\circ T^{-1})-\nabla(u_{j_k}\circ T^{-1})|^2\ d\ell\\
&=&\int_{T(M)}|\nabla(u_{j_k}-u_{j_k})\circ T^{-1}\cdot DT^{-1}(x)|^2\ d\ell\\
&=&\int_{M}|\nabla(u_{j_k}-u_{j_k})\cdot DT^{-1}(Tx)|^2\ d\ell\\
&\ge& c^2\int_M |\nabla(u_{j_k}-u_{j_k})|^2\ d\ell,
\end{eqnarray*}
where $c=\inf_{x\in M, v\in\mathbb{R}^d} |DT^{-1}(Tx)\cdot v|/|v|>0$ as $T$ is a $C^\infty$ diffeomorphism and $M$ is compact.
Now,
\begin{eqnarray}
\nonumber\lefteqn{(1+c^2)\|\nabla u_{j_k}-\nabla u_{j_l}\|^2}\\
\nonumber&\le& \|\nabla u_{j_k}-\nabla u_{j_l}\|^2+\|\nabla (\P u_{j_k})-\nabla (\P u_{j_l})\|^2\\
\nonumber&=&2\left(\|\nabla u_{j_k}\|^2+\|\nabla (\P u_{j_k})\|^2\right)+2\left(\|\nabla u_{j_l}\|^2+\|\nabla (\P u_{j_l})\|^2\right)-\left(\|\nabla(u_{j_k}+u_{j_l})\|^2+\|\nabla(\P(u_{j_k}+u_{j_l}))\|^2\right)\\
\label{expansion}
&\le&2\left(\|\nabla u_{j_k}\|^2+\|\nabla (\P u_{j_k})\|^2\right)+2\left(\|\nabla u_{j_l}\|^2+\|\nabla (\P u_{j_l})\|^2\right)-2I\|(u_{j_k}+u_{j_l})\|^2
\end{eqnarray}
By construction, as $j_k,j_l\to\infty$, the first two terms of (\ref{expansion}) both converge to $4I$, and the final term of (\ref{expansion}) converges to $8I$.
Thus, the $u_{j_k}$ form a Cauchy sequence in $W^{1,2}$, and converge to $\bar{u}$ in $W^{1,2}$.
%
\end{proof}


Because $F$ and $G$ are both\footnote{Showing $G$ is a $C^1$ functional follows identically to the arguments above for $F$.} $C^1$ functionals on $X$, we can use the method of Lagrange multipliers, and by Lemma \ref{lemmaattain} the minimiser $\bar{u}$ satisfies the Euler-Lagrange equation $F'(\bar{u})v=\mu G'(\bar{u})v$ for some $\mu\in\mathbb{R}$ and all $v\in X$.
By the constructions in the proof of Lemma \ref{functionallemma} (ii), this equation is
\begin{equation}
\label{weakeigeneqn}
\int_M (\nabla \bar{u}\cdot\nabla v)\ d\ell+\int_{T(M)}(\nabla(\P \bar{u})\cdot \nabla(\P v))\ d\ell=-2\mu \int_M \bar{u}v\ d\ell\quad\mbox{for all $v\in X$}.
\end{equation}
If we set $\lambda=\mu$, we have exactly the statement (\ref{weakeqn}).
In fact, putting  $v=\bar{u}$ we get $I=F(\bar{u})=(1/2)\int_M |\nabla \bar{u}|^2 +|\nabla(\P \bar{u})|^2\ d\ell=-\lambda\int_M \bar{u}^2\ d\ell=-\lambda$.
Thus we could have defined $\lambda$ by the Rayleigh quotient
$$-\lambda=\inf_{u\in X}\left(\frac{\int_M |\nabla u|^2\ d\ell +\int_{T(M)}|\nabla(\P u)|^2\ d\ell}{2\int_M u^2\ d\ell}\right).$$
From now on we denote $(\lambda,\bar{u})$ by $(\lambda_1,u_1)$ and search for other solution pairs.
Note that $-\lambda_1=F(u_1)\ge 0$, but that $u_1\equiv 1$ yields $F(u_1)=0$ by volume-preservation of $T$;  thus $\lambda_1=0$.
\begin{lemma}
\label{orthogonality}
If $(\lambda_1,u_1)$ and $(\lambda_2,u_2)$ are solution pairs for (\ref{weakeqn}) with $\lambda_1\neq\lambda_2$ then $\langle u_1,u_2\rangle =0$;  that is $u_1,u_2$ are orthogonal in the $L^2$ inner product.
\end{lemma}
\begin{proof}
Put $u=u_1, v=u_2$ in (\ref{weakeqn}), then put $u=u_2, v=u_1$ in (\ref{weakeqn}) and subtract the two equations to get $(\lambda_1-\lambda_2)\int_M u_1u_2\ d\ell=0$
\end{proof}
One may now follow the standard procedure for constructing a sequence of eigenvalues $0=\lambda_1> \lambda_2>\cdots$ (e.g.\ \cite{mcowen} pp212--213) by first defining
\begin{equation}
\label{rayleigh2}
-\lambda_2=\inf_{u\in X,\langle u,u_1\rangle=0}\frac{\int_M |\nabla u|^2\ d\ell+\int_{T(M)}|\nabla(\P u)|^2\ d\ell}{2\int_M u^2\ d\ell},
\end{equation}
and then inductively adding the constraint $\langle u,u_2\rangle=0$ in the next infimum to define $\lambda_3$ and so on.
The functions $u_1,u_2,\ldots,$ constructed in this way are scaled to form an orthonormal set in $L^2$.
\begin{lemma}
\label{infinitelemma}
The sequence $\lambda_n$ tends to $-\infty$ and for each $n$, the dimension of the solution space is finite.
\end{lemma}
\begin{proof}
Let $u_n,u_m$ be solutions to (\ref{rayleigh2}) corresponding to $\lambda_n,\lambda_m$ obtained inductively as above.
By (\ref{weakeqn}), setting $u=u_n, v=u_m$, we have
\begin{equation}
\label{infinite}
(1/2)\int_M \nabla u_n\cdot\nabla u_m\ d\ell+\int_{T(M)}\nabla(\P u_n)\cdot\nabla(\P u_m)\ d\ell=-\lambda_n\int u_nu_m\ d\ell.
\end{equation}
By Lemma \ref{orthogonality} we see that the RHS of (\ref{infinite}) is 0 if $n\neq m$ and $\lambda_n$ if $n=m$.
Thus, $$\|u_n\|_X^2=\int_M|\nabla u_n|^2 + u_n^2\ d\ell\le 2\left((1/2)\int_M|\nabla u_n|^2 \ d\ell+ \int_{T(M)}|\nabla (\P u_n)|^2\ d\ell\right)+1=-2\lambda_n+1.$$
By a standard argument, (e.g.\ \cite{mcowen} p213), we assume that $\lambda_n\nrightarrow-\infty$, thus $\|u_n\|_X$ is uniformly bounded in $n$ and by Rellich compactness one finds a Cauchy subsequence in $L^2$ and derives a contradiction by $L^2$ pairwise orthogonality of members of this subsequence.

Because $\lambda_n\to-\infty$, each $\lambda_n$ occurs only finitely many times and therefore the solution space for each $\lambda_n$ is finite-dimensional.
\end{proof}

\subsection{Ellipticity and strong solutions}
\label{sec:ellipticity}
We have established the existence of a set of solutions $(u_1,\lambda_1), (u_2,\lambda_2),\ldots,$ of (\ref{weakeqn})
and now wish to make a link between the solutions of (\ref{weakeqn}) and solutions of the strong formulation (\ref{strongeqn}).
The property of ellipticity of $\hat{\triangle}$ will be crucial.
Suppose we have a second order differential operator
\begin{equation}
\label{2ndorderop}
L=\sum_{i,j=1}^d a_{ij}(x)\frac{\partial^2}{\partial x_i\partial x_j}+\sum_{i=1}^d b_i(x)\frac{\partial}{\partial x_i}+c(x),\quad x\in \mathring{M},
\end{equation}
with coefficient functions $a_{ij}, b_i, c$ that are $C^\infty$ on $M$.
We will say that $a_{ij}$ satisfies uniform ellipticity if
\begin{equation}
\label{ellipticity}
\sum_{i,j=1}^d a_{ij}(x)\xi_i\xi_j\ge \gamma|\xi|^2,\qquad\mbox{for all $x\in M$, $\xi\in \mathbb{R}^d$}.
\end{equation}
\begin{lemma}
\label{ellipticlemma}
$\hat{\triangle}$ satisfies uniform ellipticity.
\end{lemma}
\begin{proof}
We have $L=\hat{\triangle}=\triangle+\P^*\triangle\P$.
Clearly $\triangle$ is elliptic, with $a_{ij}(x)=\delta_{ij}$ for all $x\in M$.
We show that $\P^*\triangle\P$ is elliptic and the result follows.
In fact, since $\P^*$ is merely composition with $T$, we only need show that $\triangle\P$ is elliptic.

Let $x_1,\ldots,x_d$ denote the standard coordinate system on $M$ in which $\triangle f=\sum_{i=1}^d \frac{\partial^2 f}{\partial x_i^2}$.
Let $f:M\to\mathbb{R}$ be $C^2$;  denoting $T^{-1}(x_1,\ldots,x_d)=(T^{-1}_1(x_1,\ldots,x_d),\ldots,T^{-1}_d(x_1,\ldots,x_d))$, and applying the chain rule for partial differentiation, one has
\begin{eqnarray*}
\lefteqn{
\frac{\partial^2}{\partial x_ix_j}\left(f\circ T^{-1}\right)}\\
&=&\left[\frac{\partial T^{-1}_1}{\partial x_j},\ldots,\frac{\partial T^{-1}_d}{\partial x_j}\right]\cdot H(f)\circ T^{-1}\cdot\left[\frac{\partial T^{-1}_1}{\partial x_i},\ldots,\frac{\partial T^{-1}_d}{\partial x_i}\right]^\top+(\nabla f)\circ T^{-1}\cdot\left[\frac{\partial^2 T^{-1}_1}{\partial x_ix_j},\ldots,\frac{\partial^2 T^{-1}_d}{\partial x_ix_j}\right]^\top,
\end{eqnarray*}
where $H(f)$ is the Hessian for $f$.
As $\triangle(\P f)=\sum_{i=1}^d \frac{\partial^2}{\partial x_i^2}(f\circ T^{-1})$, the representation of $\triangle (\P f)$ in the form (\ref{2ndorderop}) is
\begin{equation}
\label{2ndorderopspecific}
\sum_{i,k,l=1}^d \frac{\partial T^{-1}_k}{\partial x_i}\frac{\partial T^{-1}_l}{\partial x_i}[H(f)]_{kl}\circ T^{-1}+\sum_{i,k=1}^d \frac{\partial^2 T^{-1}_k}{\partial x_i^2}[\nabla f]_k\circ T^{-1}.
\end{equation}
Since $T$ is a $C^\infty$ diffeomorphism, both $a_{kl}=\sum_{j=1}^d\frac{\partial T^{-1}_k}{\partial x_j}\frac{\partial T^{-1}_l}{\partial x_j}$ and $b_i=\sum_{j=1}^d\frac{\partial^2 T^{-1}_i}{\partial x_j^2}$ are $C^\infty$ and bounded as a function of $x\in T(M)$ for each $k,l=1,\ldots,d$.
In order to show that $a_{kl}$ is uniformly elliptic, we note that $a_{kl}(x)$ is the inner product of the $k^{th}$ and $l^{th}$ rows of the Jacobian matrix $D(T^{-1})(x)$.
Thus, for each $x\in T(M)$, $a_{kl}(x)$ is a Gram matrix, formed from the $d$ (linearly independent, because $T$ is a diffeomorphism) rows of $D(T^{-1})(x)$, denoted $r_1,\ldots,r_d$.
The matrix $a_{kl}(x)$ is symmetric and therefore positive definite if and only if all of its eigenvalues are positive (e.g.\ Theorem 7.2.1 \cite{hornjohnson}).
Using the structure of the Gram matrix, we know $a_{kl}(x)$ is positive semidefinite and is nonsingular if and only if $\{r_1,\ldots,r_d\}$ are linearly independent (e.g.\ Theorem 7.2.10 \cite{hornjohnson}).
Thus, $a_{kl}(x)$ is positive definite and satisfies $\sum_{k,l=1}^d a_{kl}(x)\xi_k\xi_l\ge \gamma(x)|\xi|^2$ for some $\gamma(x)>0$ for every $x\in M$.
By compactness of $M$, we can set $\gamma=\min_{x\in M}\gamma(x)$.
 \end{proof}

We show that solutions of (\ref{weakeqn}) are in $C^\infty(M)$ and satisfy (\ref{strongeqn})--(\ref{strongbc}).
By the arguments used to obtain Corollary 8.4.1 \cite{jost} or the discussion on p.214 \cite{GilbargTrudinger}, provided that our second-order differential operator $\hat{\triangle}$ is uniformly elliptic and that $T$ and $M$ are $C^\infty$, one has that a solution $u$ of (\ref{weakeqn}) is in fact $C^\infty$ on $M$.
Following the arguments in \S8.4--8.5 \cite{jost}, because $u\in C^\infty(M)$ we
can apply Green's first identity to the first term on the LHS of (\ref{weakeqn}) to obtain:
\begin{equation}
\label{green1}
\int_M \nabla v\cdot \nabla u\ d\ell=-\int_M v (\triangle u)\ d\ell+\int_{\partial M} v (\nabla u)\cdot \mathbf{n}\ d\ell_{d-1}.
\end{equation}
Now, the second term on the LHS of (\ref{weakeqn}):  denoting $\tilde{\mathbf{n}}(y)$ to be the outward unit normal at $y\in T(M)$, and using  Green's first identity and change of variables under $T$:
\begin{eqnarray*}
\int_{T(M)} \nabla (\P v)\cdot \nabla (\P u)\ d\ell&=&-\int_{T(M)} \P v\cdot \triangle \P u\ d\ell+\int_{\partial T(M)} \P v \left[\nabla (\P u)\right]\cdot \tilde{\mathbf{n}}\ d\ell_{d-1}\\
&=&-\int_M v\cdot(\P^*\triangle\P)u\ d\ell+\int_{\partial T(M)} \P v \left[\nabla (\P u)\right]\cdot \tilde{\mathbf{n}}\ d\ell_{d-1}.
\end{eqnarray*}
We now manipulate the second term above using the chain rule for differentiation, change of variables under $T$, and volume-preservation of $T$: 
\begin{eqnarray}
\nonumber
\int_{\partial T(M)} \P v \left[\nabla (\P u)\right]\cdot \tilde{\mathbf{n}}\ d\ell_{d-1}&=&\int_{\partial T(M)} v\circ T^{-1} \left[\left((\nabla u)\circ T^{-1}\right)^\top\cdot DT^{-1}\right]\cdot \tilde{\mathbf{n}}\ d\ell_{d-1}\\
\label{halfwayeqn}&=&\int_{\partial M} v \left[\left(\nabla u\right)^\top \cdot (DT^{-1}\circ T)\right]\cdot (\tilde{\mathbf{n}}\circ T)|\det DT_{|\mathcal{T}_x(\partial M)}|\ d\ell_{d-1}.
\end{eqnarray}
Note that $\tilde{\mathbf{n}}(Tx)=(DT(x)^{-1})^\top \mathbf{n}(x)/\|(DT(x)^{-1})^\top \mathbf{n}(x)\|$.
By Lemma \ref{sl}, $|\det DT_{|\mathcal{T}_x(\partial M)}|=\|(DT(x)^{-1})^\top \mathbf{n}(x)\|$.
Thus,
$$
(\ref{halfwayeqn})=\int_{\partial M} v \left[\left(\nabla u\right)^\top \cdot(DT)^{-1} \right]\cdot \left((DT)^{-1}\right)^\top\cdot\mathbf{n}\ d\ell_{d-1},
$$
and we arrive at the transformed version of (\ref{weakeqn}):
\begin{eqnarray}
\label{weakeqnsmooth1}
\lefteqn{(1/2)\int_M v\cdot(\triangle+\P^*\triangle\P)u\ d\ell=\lambda \int_M v\cdot u\ d\ell}\\
\label{weakeqnsmooth2}&&\quad+(1/2)\int_{\partial M} v (\nabla u)\cdot \mathbf{n}\ d\ell_{d-1}+(1/2)\int_{\partial M} v \left[\left(\nabla u\right)^\top \cdot(DT)^{-1} \right]\cdot \left((DT)^{-1}\right)^\top\cdot\mathbf{n}\ d\ell_{d-1}, \forall v\in W^{1,2}.
\end{eqnarray}
By considering all $v\in W^{1,2}_0$ (the closure of $C^\infty_0(\mathring{M})\cap W^{1,2}(M)$ with respect to $\|\cdot\|_{W^{1,2}(M)}$) in (\ref{weakeqnsmooth1})--(\ref{weakeqnsmooth2}) we see that $(1/2)(\triangle+\P^*\triangle\P)u=\lambda u$ on $\mathring{M}$ (let $f=(1/2)(\triangle+\P^*\triangle\P)u-\lambda u$ and WLOG suppose $f(x)>0$ at some $x\in\mathring{M}$. Necessarily, $f(x)>0$ for $x$ in some open $O\subset \mathring{M}$ and consider a bump function $v$ positive in a ball contained in $O$ and zero outside $O$ to derive a contradiction).
Now (\ref{weakeqnsmooth2}) implies that
$$\int_{\partial M} v (\nabla u)\cdot \mathbf{n}\ d\ell_{d-1}+\int_{\partial M} v \left[\nabla u \cdot (DT)^{-1}\cdot \left((DT)^{-1}\right)^\top\right]\cdot\mathbf{n}\ d\ell_{d-1}=0\quad \forall v\in W^{1,2}.$$
Again using the fact that $u\in C^\infty(M)$, by an argument on $\partial M$ analogous to the parenthetical argument above, it follows that (\ref{strongbc}) holds.

\subsection{Proof of Theorem \ref{objthm}}

A key component to the proof of Theorem \ref{objthm} is the fact that the Laplacian commutes with isometries.
\begin{lemma}
\label{isometrylemma}
\begin{equation}
\label{isometryeqn}
\triangle(f\circ \Phi_{t_0})=(\triangle f)\circ\Phi_{t_0}.
\end{equation}
\end{lemma}
\begin{proof}
By (\ref{2ndorderopspecific}), and the fact that $\Phi_{t_0}$ is affine, the LHS of (\ref{isometryeqn}) is
\begin{eqnarray*}
\sum_{i,k,l=1}^d Q(t_0)_{ki}\cdot[H(f)]_{kl}\circ \Phi_{t_0} \cdot Q(t_0)_{li}
&=&\Tr(Q(t_0)^\top\cdot [H(f)]\circ \Phi_{t_0}\cdot Q(t_0))\\
&=&\Tr([H(f)]\circ \Phi_{t_0}),
\end{eqnarray*}
using orthogonality of $Q(t_0)$, to obtain exactly the RHS of (\ref{isometryeqn}).
\end{proof}

\begin{proof}[Proof of Theorem \ref{objthm}]
We first show equivalence of (\ref{origevalprob}) and (\ref{transfevalprob}).
\begin{eqnarray*}
\dot{\hat{\triangle}}&=&\triangle+\P^*_{\dot{T}}\triangle \P_{\dot{T}}\\
&=&\triangle+(\P_{\Phi_{t_1}}\circ \P\circ \P_{\Phi_{t_0}}^{-1})^*\triangle(\P_{\Phi_{t_1}}\circ \P\circ \P_{\Phi_{t_0}}^{-1})\\
&=&\triangle+\P_{\Phi_{t_0}}\P^*\triangle\P \P_{\Phi_{t_0}}^{-1},
\end{eqnarray*}
where we have used Lemma \ref{isometrylemma} and the fact that $\P_{\Phi_{t_0}}$ and $\P_{\Phi_{t_1}}$ are unitary operators.
Now, if $\hat{\triangle}f=\lambda f$,
\begin{eqnarray*}
\dot{\hat{\triangle}}\P_{\Phi_{t_0}}f&=&(\triangle \P_{\Phi_{t_0}}+\P_{\Phi_{t_0}}\P^*\triangle\P)f\\
&=&\P_{\Phi_{t_0}}(\triangle +\P^*\triangle\P)f\\
&=&\lambda\P_{\Phi_{t_0}}f,
\end{eqnarray*}
as required, where we have again used Lemma \ref{isometrylemma}.

We now demonstrate equivalence of (\ref{strongorigbc}) and (\ref{strongtransfbc}).
We calculate each of the components of (\ref{strongtransfbc}).
First, note that $\nabla (\P_{\Phi_{t_0}}f)(x)=\left[(\nabla f)\circ\Phi_{t_0}^{-1}(x)\right]^\top\cdot Q(t_0)^{-1}$.
Next, $D\dot{T}(x)=Q(t_1)\cdot DT(\Phi^{-1}_{t_0}(x))\cdot Q(t_0)^{-1}$, so $D\dot{T}(x)^{-1}=Q(t_0)\cdot DT(\Phi^{-1}_{t_0}(x))^{-1}\cdot Q(t_1)^{-1}$ and $\left(D\dot{T}(x)^{-1}\right)^\top=Q(t_1)\cdot \left(DT(\Phi^{-1}_{t_0}(x))^{-1}\right)^\top\cdot Q(t_0)^{-1}$.
Thus, for $x\in \partial(\Phi_{t_0}(M))$,
\begin{eqnarray*}
\lefteqn{(\ref{strongtransfbc})}\\
&=&(\nabla f)\circ\Phi_{t_0}^{-1}(x)\cdot Q(t_0)^{-1}\cdot[Q(t_0)\mathbf{n}(\Phi_{t_0}^{-1}x)+Q(t_0)\cdot DT(\Phi_{t_0}^{-1}x)^{-1}\cdot \left(DT(\Phi_{t_0}^{-1}x)^{-1}\right)^\top \mathbf{n}(\Phi_{t_0}^{-1}x)]\\
&=&(\nabla f)\circ\Phi_{t_0}^{-1}(x)\cdot[\mathbf{n}(\Phi_{t_0}^{-1}x)+DT(\Phi_{t_0}^{-1}x)^{-1}\cdot \left(DT(\Phi_{t_0}^{-1}x)^{-1}\right)^\top \mathbf{n}(\Phi_{t_0}^{-1}x)],
\end{eqnarray*}
which is exactly condition (\ref{strongorigbc}) evaluated at $\Phi_{t_0}^{-1}x\in \partial M$.
\end{proof}

\section{Proof of Theorem \ref{analcvgce}}
\label{sec:analcvgce}
The crux of the proof of Theorem \ref{analcvgce} is linking the diffusion operators $\mathcal{D}_{M,\epsilon}, \mathcal{D}_{T(M),\epsilon}$ with $\triangle$.
This linking is possible because of the symmetry of the smoothing kernel $q_\epsilon$.
At small scales (small $\epsilon$), $\mathcal{D}_{M,\epsilon}, \mathcal{D}_{T(M),\epsilon}$ are close to the identity operator, and because of the spatial symmetry of $q_\epsilon$, the next dominant term depends on second order derivatives.

\begin{lemma}
\label{nddifflemma}
Let $M$ be a connected, compact Riemannian manifold of vanishing curvature, and $f:M \to\mathbb{R}$ be $C^3$.
Let $q:M\to\mathbb{R}^+$ be a nonnegative density with compact support, with mean the origin, and covariance matrix $c\cdot I$, where $I$ is the $d\times d$ identity matrix.
We scale $q$ to form $q_\epsilon(x)=q(x/\epsilon)/\epsilon^d$, and define $\D f(x)=\int_M q_\epsilon(x-y)f(y)\ d\ell(y)$.
Then
\begin{equation}
\label{diffform1}\lim_{\epsilon\to 0} \frac{(\D-I)f(x)}{\epsilon^2}=(c/2)\triangle f(x),
\end{equation}
for each $x\in \mathring{M}$.
\end{lemma}
\begin{proof}
Using Taylor, we expand $f$ in an $\epsilon$-ball about $x$
\begin{equation}
\label{taylor}
f(x+y)=\sum_{|\alpha|=0}^2\frac{D^\alpha f(x)}{|\alpha| !} y^\alpha+\sum_{|\alpha|=3} R_\alpha(x+y)y^\alpha,
\end{equation}
where $R_\alpha(x+y)$ is the remainder term.
The notation used is $\alpha=(\alpha_1,\ldots,\alpha_d)$, where $\alpha_i$ is the number of derivatives in coordinate direction $x_i$;  $|\alpha|$ denotes the sum of the elements of $\alpha$ and $\alpha !=\prod_{i=1}^d \alpha_i !$.
If $\partial M\neq \emptyset$, then put $\epsilon'=\inf_{z\in\partial M}\dist(x,z)$.
In the following we assume that $\epsilon\le \epsilon'$.
\begin{eqnarray}
\nonumber\D f(x)&=&\int_M q_\epsilon(x-y)f(y)\ d\ell(y)\\
\nonumber&=&\int_M q_\epsilon(y)f(x+y)\ d\ell(y) \\
\label{Reqn}&=&\int_M q_\epsilon(y)\left[\sum_{|\alpha|=0}^2\frac{D^\alpha f(x)}{|\alpha| !} y^\alpha+\sum_{|\alpha|=3} R_\alpha(x+y)y^\alpha\right]\ d\ell(y)
\end{eqnarray}
The terms in the first sum of order $|\alpha|=0, 1, 2$, respectively, are: $f(x)$, 0, $\sum_{|\alpha|=2} m_\alpha(q_\epsilon) D^\alpha f(x)/\alpha !$,
where $m_\alpha(q_\epsilon)=\int_M q_\epsilon(y) y^\alpha\ d\ell(y)$ denotes the tensor of $\alpha$-moments of $q_\epsilon$.
The order 1 term is zero because of the assumption on the mean (the $|\alpha|=1$ moments) of $q$.
We note that $m_\alpha(q_\epsilon)=\int_M q_\epsilon(y) y^\alpha\ d\ell(y)=\int_M q(y/\epsilon)/\epsilon^d y^\alpha\ d\ell(y)=\int_M q(z)\epsilon^{|\alpha|}z^\alpha\ d\ell(z)=\epsilon^{|\alpha|}m_\alpha(q)$.
We can further simplify the order 2 term as $\sum_{|\alpha|=2} m_\alpha(q_\epsilon) D^\alpha f(x)/\alpha !=(c/2)\epsilon^2\triangle f(x)$  using the fact that the covariance matrix of $q$ is $c\times I$.
Rearranging (\ref{Reqn}), we have
$$\D f(x)-f(x)-(c/2)\epsilon^2\triangle f(x)
=\int_M q_\epsilon(y)\sum_{|\alpha|=3}R_\alpha(x+y)y^\alpha\ d\ell(y).$$
For $|\alpha|=3$, and $y\in B_\epsilon(x)$, one has $|R_\alpha(x+y)|\le (1/\alpha !)\max_{|\alpha|=3}\max_{z\in B_\epsilon(x)}|D^\alpha f(z)|=:C(x)$, so
$$|\D f(x)-f(x)-(c/2)\epsilon^2\triangle f(x)|\le C(x)\epsilon^3 \sum_{|\alpha|=3} m_\alpha(q).$$ 

\end{proof}

\begin{proof}[Proof of Theorem \ref{analcvgce}]
By Lemma \ref{nddifflemma}, we have $\D f(x)=f(x)+\epsilon^2(c/2)\triangle f(x)+O(\epsilon^3)$, where $O(\epsilon^3)$ means the error term is of the form $\epsilon^3 \mathcal{R}(x)$.
Therefore
$\P \D f(x)=\P f(x)+\epsilon^2(c/2)\P \triangle f(x)+O(\epsilon^3)$.
Since $T$ is $C^3$ we may again apply Lemma \ref{nddifflemma} to obtain
\begin{eqnarray*}
\D\P\D f(x)&=&\D\left(\P f(x)+\epsilon^2(c/2)\P \triangle f(x)+O(\epsilon^3)\right)\\
&=&\P f(x)+\epsilon^2((c/2)\P\triangle f(x)+(c/2)\triangle\P f(x))+O(\epsilon^3)
\end{eqnarray*}
Finally,
\begin{eqnarray*}
\D^*\P^*\D^*\D\P\D f(x)&=&\P^*\P f(x)+\epsilon^2(c/2)(\P^*\P\triangle+\P^*\triangle \P+\P^*\triangle \P+\triangle\P^*\P)f(x)+O(\epsilon^3)\\
&&=f(x)+\epsilon^2 c(\triangle+\P^*\triangle \P)f(x)+O(\epsilon^3)
\end{eqnarray*}
\end{proof}

\begin{remark}
\label{4thorderremark}
It is reasonably natural for $q$ to have additional symmetry, so that $m_\alpha(q)=0$ for $|\alpha|=3$.
In this case, the error term in the above proof is $O(\epsilon^4)$.
For example, the uniform diffusion on a unit ball considered in Example \ref{unifexample} has $m_\alpha(q)=0$ for all $|\alpha|$ odd.
\end{remark}

\bibliographystyle{plain}

\end{document}